\documentclass[a4paper]{amsart} 
\usepackage{amsmath,amsfonts,amssymb,graphicx,url,amsthm,hyperref}
\usepackage[nocompress]{cite}
\usepackage{xcolor}
\usepackage{mathrsfs}
\usepackage[noend]{algpseudocode}
\usepackage[ruled,vlined,norelsize,algosection]{algorithm2e}
\usepackage[T1]{fontenc}
\newcommand{\id}{\mathsf{id}}
\renewcommand{\O}{\Omega}

\newcommand{\dom}{\mathcal{D}}
\newcommand{\kernel}{\kappa}
\newcommand{\cK}{\mathcal K}

\renewcommand{\d}{\operatorname{d}\!}
\newcommand{\bm}{\boldsymbol}
\DeclareMathOperator{\diam}{diam}

\DeclareMathOperator{\spn}{span}

\DeclareMathOperator\dist{dist}
\DeclareMathOperator\supp{supp}

\usepackage{caption}
\usepackage{subcaption}
\usepackage{tikz,pgfplots,pgfplotstable}
\pgfplotsset{compat=newest}
\usetikzlibrary{arrows,snakes,backgrounds,positioning,calc}
\usepackage{graphicx,color,nicefrac} 
\usepackage{amssymb, amsmath, amsbsy,mathtools}
\usepackage{amsmath,amsfonts}
\usepackage[normalem]{ulem}
\usepackage[export]{adjustbox}
\newtheorem{theorem}{Theorem}[section]
\newtheorem{lemma}[theorem]{Lemma}
\newtheorem{corollary}[theorem]{Corollary}
\newtheorem{proposition}[theorem]{Proposition}
\newtheorem{definition}[theorem]{Definition}

\newtheorem{remark}[theorem]{Remark}

\def\letters{a,b,c,d,e,f,g,h,i,j,k,l,m,n,o,p,q,r,s,t,u,v,w,x,y,z}
\def\Letters{A,B,C,D,E,F,G,H,I,J,K,L,M,N,O,P,Q,R,S,T,U,V,W,X,Y,Z}
\makeatletter
\@for \@l:=\Letters \do{%
  \expandafter\edef\csname\@l bb\endcsname{%
  \noexpand\ensuremath{\noexpand\mathbb{\@l}}}%
  \expandafter\edef\csname\@l bf\endcsname{%
  {\noexpand\bf \@l}}%
  \expandafter\edef\csname\@l cal\endcsname{%
  \noexpand\ensuremath{\noexpand\mathcal{\@l}}}%
  \expandafter\edef\csname\@l eu\endcsname{%
  \noexpand\ensuremath{\noexpand\EuScript{\@l}}}%
  \expandafter\edef\csname\@l frak\endcsname{%
  \noexpand\ensuremath{\noexpand\mathfrak{\@l}}}%
  \expandafter\edef\csname\@l rm\endcsname{%
  {\noexpand\rm \@l}}%
  \expandafter\edef\csname\@l scr\endcsname{%
  \noexpand\ensuremath{\noexpand\mathscr{\@l}}}%
}
\@for \@l:=\letters \do{%
  \expandafter\edef\csname\@l bf\endcsname{%
  {\noexpand\bf \@l}}%
  \expandafter\edef\csname\@l frak\endcsname{%
  \noexpand\ensuremath{\noexpand\mathfrak{\@l}}}%
  \expandafter\edef\csname\@l scr\endcsname{%
  \noexpand\ensuremath{\noexpand\mathscr{\@l}}}%
}
\makeatother
 \definecolor{shadecolor}{rgb}{0.6, 0.6, 0.6} 
  \definecolor{darkgreen}{rgb}{0, 0.6, 0}
\newcommand{\isdef}{\mathrel{\mathrel{\mathop:}=}}
\newcommand{\defis}{\mathrel{=\mathrel{\mathop:}}}
\begin{document}
\title{Multiresolution kernel matrix algebra}
\author{H.~Harbrecht}
\address{H.~Harbrecht,
Departement f\"ur Mathematik und Informatik, 
Universit\"at Basel, 
Spiegelgasse 1, 4051 Basel, Schweiz.}
\email{helmut.harbrecht@unibas.ch}
\author{M.~Multerer}
\address{
M.~Multerer,
Istituto Eulero,
USI Lugano,
Via la Santa 1, 6962 Lugano, Svizzera.}
\email{michael.multerer@usi.ch}
\author{O.~Schenk}
\address{
O.~Schenk,
Institute of Computing,
USI Lugano,
Via la Santa 1, 6962 Lugano, Svizzera.}
\email{olaf.schenk@usi.ch}
\author{Ch.~Schwab}
\address{
Ch.~Schwab,
Seminar für Angewandte Mathematik,
ETH Z\"urich,
R\"amistrasse 101, 8092 Z\"urich, Schweiz.}
\email{christoph.schwab@sam.math.ethz.ch}
\pgfplotsset{select coords between index/.style 2 args={
    x filter/.code={
        \ifnum\coordindex<#1\def\pgfmathresult{}\fi
        \ifnum\coordindex>#2\def\pgfmathresult{}\fi
    }
}}
\begin{abstract}
We propose a sparse algebra for samplet compressed kernel matrices,
to enable efficient scattered data analysis.
We show the compression of kernel matrices by means of samplets 
produces optimally sparse
matrices in a certain S-format. 
It can be performed in cost and memory that scale essentially linearly
with the matrix size $N$,
for kernels of finite differentiability, along with
addition and multiplication of S-formatted matrices.
We prove and exploit the fact that 
the inverse of a kernel matrix (if it exists) is compressible in 
the S-format as well. 
Selected inversion allows to directly compute the 
entries in the corresponding sparsity pattern.
The S-formatted matrix operations enable the efficient, approximate 
computation of more complicated matrix functions such as 
${\bm A}^\alpha$ or $\exp({\bm A})$.
The matrix algebra is justified mathematically by
pseudo differential calculus.
As an application,
efficient Gaussian process learning algorithms for
spatial statistics is considered. 
Numerical results are presented
to illustrate and quantify our findings.
\end{abstract}

\maketitle
\section{Introduction}\label{sec:intro}
The concept of \emph{samplets} has been introduced in \cite{HM21}
by abstracting the wavelet construction from \cite{TW03} 
to general discrete data sets in euclidean space.
A samplet basis is a multiresolution 
analysis of discrete signed measures, where stability 
is entailed by the orthogonality of the basis. 
Samplets are data-centric and can be constructed 
such that their measure integrals vanish for polynomials
up to a predefined degree. 
Thanks to this \emph{vanishing moment property in ambient space}, 
kernel matrices, as they arise in scattered data approximation,
become quasi-sparse in the samplet basis. This means that these
kernel matrices are
\emph{compressible} in samplet coordinates, \emph{$S$-compressible} 
for short, and can be replaced 
by sparse matrices. We call the resulting sparsity pattern the
\emph{compression pattern}. 
The latter has been characterized in \cite[Section~5.3]{HM21}. 
Given a quasi-uniform data set of cardinality $N$,
i.e., the distance between neighboring points is
uniformly bounded from below and above by $N^{-1/d}$ with
$d \geq 1$ being the spatial dimension of the data,
the $S$-compressed kernel matrix contains only
$\mathcal{O}(N\log N)$ relevant entries, 
for kernels of possibly low regularity.

In this article, we develop \emph{fast arithmetic operations}
for $S$-compressed kernel matrices. 
By fixing the sparsity
pattern, we can perform addition and multiplication of kernel 
matrices with high precision in essentially linear cost. 
The derived cost bounds assume quasi-uniformity of the data points.
Even so, all algorithms can still be applied 
if the quasi-uniformity assumption does not hold. In this
case, however, the established cost bounds may become invalid.
Similar approaches for realizing arithmetic
operations of nonlocal operators exist by means hierarchical matrices,
see \cite{Hackbusch,DHM19,Boe10,HKT08,GHK03}, and by means of wavelets,
see \cite{SchW06,BM02,B96}.

We prove that the inverses of (regularized) kernel matrices
are compressible with respect to the original compression pattern.
We can thus employ the selected inversion algorithm proposed 
in \cite{selinv}, to efficiently approximate the inverse.
Our concrete implementation is based on a supernodal left-looking
LDLT-factorization of the underlying matrix, which is available in
the sparse, direct solver \texttt{Pardiso}, see~\cite{PARDISO:2022}. 
The selected inversion computes (in the absence of rounding) the
exact matrix inverse of the $S$-compressed matrix on its matrix
pattern. Likewise, matrix addition and matrix multiplication are
performed exactly on the prescribed compression pattern.
This means that, the relevant matrix coefficients are 
computed exactly when adding, multiplying, and inverting 
$S$-compressed kernel matrices. The only error introduced is the
matrix compression error issuing from the restriction to the
compression pattern.

Having a fast formatted matrix addition and fast matrix inversion 
at hand enables the fast approximate evaluation of 
holomorphic operator functions via contour integrals in order 
to derive more complicated matrix functions.
This has been envisioned in \cite{BM02}
(``We conjecture and provide numerical evidence that functions
of operators inherit this property'') and suggested in \cite{Hale2008}. 
In the present paper we prove, 
using the samplet algebra, that, up to (exponentially small) contour
quadrature errors, these contour integrals are 
computed exactly on the prescribed pattern. 
This is in contrast to previously proposed methods. 
In addition, many applications particularly require
the computation of a subset of the elements of a given matrix inverse.
Important examples are sparse inverse covariance matrix estimation
in $\ell^1$-regularized Gaussian maximum likelihood estimation, see
\cite{HSDRP,BESS}, or integrated nested Laplace approximations
for approximate Bayesian inference, cp.\ \cite{RINLA:2021}.
Other examples of computing a subset of the inverse are electronic
structure calculations of materials utilizing multipole 
expansions, where the diagonal and occasionally sub-diagonals 
of the discrete Green's function are required to determine 
the electron density \cite{Darve2008,LYYE}.

We provide a rigorous
theoretical underpinning of the algorithms under consideration
by means of pseudodifferential 
calculus \cite{Taylor81,HorIII}. 
To this end,
we focus on kernels of reproducing kernel 
Hilbert spaces and assume that the associated 
integral operators 
correspond, via the Schwarz kernel theorem, 
to classical, elliptic pseudodifferential operators,
from the H\"ormander class $S^m_{1,0}$, cp.\ \cite{HorIII}. 
A prominent example of such kernels is the Mat\'ern class of kernels,
see \cite{MAT}, also called Sobolev splines \cite{FY11}. 
The latter are known to generate the Sobolev spaces of positive 
order, and correspond to fractional powers of the shifted 
Laplacian.
We prove that such 
pseudodifferential operators are compressible in 
samplet coordinates, 
meaning that for numerical representation,
only the coefficients in the 
associated compression pattern need to be computed. 
Admissible classes comprise in particular the smooth
H\"ormander class $S^m_{1,0}$,
but also considerably larger 
kernel classes of finite smoothness,
which admit Calderon-Zygmund estimates and
an appropriate operator calculus, see, e.g., \cite{Abels,Taylor91}.
The corresponding operator calculus implies that 
sums, concatenations, powers and holomorphic functions 
of self-adjoint, elliptic pseudodifferential operators
yield again pseudodifferential operators.
As a consequence the corresponding
operations on kernel matrices in samplet coordinates 
result again in compressible matrices.

The rest of this article is structured as follows. 
In Section \ref{section:RKHS}, we briefly introduce the
scattered data framework under consideration and recall
the relevant theory for reproducing kernel Hilbert spaces.
The construction of samplets and the samplet matrix 
compression from \cite{HM21} are summarized in
Section~\ref{sec:compression}.
The main contribution of this article is Section~\ref{sec:arithmetic}.
Here, we develop and analyze arithmetic operations for 
compressed kernel matrices in samplet coordinates. In 
Section~\ref{sec:results}, we perform numerical experiments 
in order to qualify and quantify the matrix algebra. 
Beyond benchmarking experiments, we consider here
the computation of an implicit surface
from scattered data using Gaussian process learning.
Finally, 
the required details from the theory pseudodifferential operators, 
especially the associated calculus,
are collected in Appendix~\ref{sec:pseudos}.

Throughout this article, in order to avoid the repeated use of
generic but unspecified constants, by \(C\lesssim D\)
we indicate that $C$ can be bounded by a multiple of $D$,
independently of parameters which $C$ and $D$ may depend on.
Moreover, \(C\gtrsim D\) is defined as \(D\lesssim C\)
and \(C\sim D\) as \(C\lesssim D\) and \(D\lesssim C\).

\section{Reproducing kernel Hilbert spaces}
\label{section:RKHS}
Let \((\Hcal,\langle\cdot,\cdot\rangle_\Hcal)\) be a 
Hilbert space of functions \(h\colon\Omega\to\Rbb\)
with dual space \(\Hcal'\). Herein,
\(\Omega\subset\Rbb^d\) is a given bounded domain or
a lower-dimensional manifold. 
Furthermore, let \(\kernel\) be a symmetric 
and positive definite (SPD) kernel, i.e.,
\([\kernel({\bm x}_i,{\bm x}_j)]_{i,j=1}^N\) is a symmetric
and positive semi-definite matrix for each \(N\in\Nbb\) and any
point selection \({\bm x}_1,\ldots,{\bm x}_N\in\Omega\).
We recall that \(\kernel\) is the reproducing kernel for \(\Hcal\),
iff \(\kernel({\bm x},\cdot)\in\Hcal\) for every
\({\bm x}\in\Omega\) and
\(h({\bm x})=\langle\kernel({\bm x},\cdot),h\rangle_\Hcal\)
for every \(h\in\Hcal\). In this case, we call 
\((\Hcal,\langle\cdot,\cdot\rangle_\Hcal)\) a reproducing
kernel Hilbert space (RKHS).

Let \(X\isdef\{{\bm x}_1,\ldots,{\bm x}_N\}\subset\Omega\) 
denote a set of $N$ mutually distinct points. 
With respect to the set \(X\), we introduce 
the subspace
\begin{equation}\label{eq:HX}
\Hcal_X\isdef
\spn\{\kernel({\bm x}_1,\cdot),\ldots,\kernel({\bm x}_N,\cdot)\}
\subset\Hcal.
\end{equation}
Corresponding to \(\Hcal_X\), we consider the subspace
\(\Xcal\isdef\spn\{\delta_{{\bm x}_1},
\ldots,\delta_{{\bm x}_N}\}\subset\Hcal'\),
which is spanned by the Dirac measures supported at the points of 
\(X\), i.e.,
\[
\delta_{{\bm x}_i}({\bm x})\isdef
\begin{cases}
1,&\text{if }{\bm x}={\bm x}_i,\\
0,&\text{otherwise}.
\end{cases}
\]
For a continuous function $f\in C(\Omega)$, we use the notation
\[
(f,\delta_{{\bm x}_i})_\Omega\isdef
\int_{\Omega}f({\bm x})\delta_{{\bm x}_i}(\d{\bm x})
=f({\bm x}_i).
\]
As the kernel \(\kernel({\bm x},\cdot)\) is the Riesz representer of
the point evaluation \((\cdot,\delta_{\bm x})_\Omega\), we particularly have 
\[
(h,\delta_{\bm x})_\Omega
=\langle\kernel({\bm x},\cdot),h\rangle_\Hcal\quad\text{for every }
h\in\Hcal.
\]
Thus, the space \(\Xcal\) is isometrically isomorphic to the subspace
$\Hcal_X$ from \eqref{eq:HX} and
we identify
\[
u=\sum_{i=1}^Nu_i\delta_{{\bm x}_i}\in\Xcal\quad\text{with}\quad
\hat{u}=\sum_{i=1}^Nu_i\kernel({\bm x}_i,\cdot)\in\Hcal_X.
\]
Later on, we endow \(\Xcal\) with the inner product 
\begin{equation}\label{eq:innerp}
\langle u,v\rangle_\Xcal\isdef\sum_{i=1}^N u_iv_i,\quad\text{where }
u=\sum_{i=1}^Nu_i\delta_{{\bm x}_i},\ v=\sum_{i=1}^Nv_i\delta_{{\bm x}_i}.
\end{equation}
This inner product is different from the
restriction of the canonical one in \(\Hcal\) to \(\Hcal_X\).
The latter is given by
\[
\langle\hat{u},\hat{v}\rangle_\Hcal={\bm u}^\intercal{\bm K}{\bm v}
\]
with the symmetric and positive semi-definite \emph{kernel matrix}
\begin{equation}\label{eq:KernelMatrix}
{\bm K}\isdef [\kernel({\bm x}_i,{\bm x}_j)]_{i,j=1}^N\in\Rbb^{N\times N}
\end{equation}
and \({\bm u}\isdef[u_i]_{i=1}^N\) and \({\bm v}\isdef[v_i]_{i=1}^N\).

A consequence of the duality between \(\Hcal_X\)
and \(\Xcal\) is that
the \(\Hcal\)-orthogonal projection
of a function \(h\in\Hcal\) onto \(\Hcal_X\) is given by the 
interpolant
\[
s_h({\bm x})\isdef\sum_{i=1}^N\alpha_i\kernel({\bm x}_i,\cdot),
\]
which satisfies $s_h({\bm x}_i) = h({\bm x}_i)$ for all ${\bm x}_i\in X$. 
The associated coefficients \({\bm\alpha}=[\alpha_i]_{i=1}^N\)
are given by the solution to the linear system
\begin{equation}\label{eq:LSE}
{\bm K}{\bm\alpha}={\bm h}
\end{equation}
with right hand side \({\bm h}=[h({\bm x}_i)]_{i=1}^N\).

From \cite[Corollary 11.33]{Wendland2004}, 
we have the following approximation result.
\begin{theorem}\label{thm:Kernelapprox}
Let \(\Omega\subset\Rbb^d\) be a bounded Lipschitz domain satisfying
an interior cone condition. Suppose that the Fourier transform of the kernel
\(\kernel({\bm x}-{\bm y})\)
satisfies
\begin{equation}\label{eq:MatFourier}
\widehat{\kernel}({\bm{\bm\xi}})\sim(1+\|{\bm{\bm\xi}}\|_2^2)^{-\tau},
\quad{\bm{\bm\xi}}\in\Rbb^d.
\end{equation}
Then for \(0\leq t < \lceil\tau\rceil-d/2-1\), the error between
\(f\in H^\tau(\Omega)\) and its interpolant
\(s_{f,X}\) satisfies the bound
\[
\|f-s_{f,X}\|_{H^{t}(\Omega)}\lesssim h_{X,\Omega}^{\tau-t}\|f\|_{H^\tau(\Omega)}
\]
for a sufficiently small \emph{fill distance}
\begin{equation}\label{eq:FillDistance}
h_{X,\Omega}
\isdef\sup_{{\bm x}\in\Omega}\min_{{\bm x}_i\in X}
\|{\bm x}-{\bm x}_i\|_2.
\end{equation}
\end{theorem}

One class of kernels satisfying the conditions of Theorem~\ref{thm:Kernelapprox}
are the \emph{isotropic Mat\'ern kernels}, 
also called \emph{Sobolev splines}, see \cite{FY11}.
These kernels play an important role in applications, such as spatial statistics
\cite{Rasmussen2006}. They are given by 
\begin{equation}\label{eq:matern}
\kernel_\nu(r)\isdef\frac{2^{1-\nu}}{\Gamma(\nu)}
\bigg(\frac{\sqrt{2\nu}r}{\ell}\bigg)^\nu
K_\nu\bigg(\frac{\sqrt{2\nu}r}{\ell}\bigg)
\end{equation}
with \(r\isdef\|{\bm x}-{\bm y}\|_2\), 
smoothness parameter $\nu>0$ and length scale parameter 
$\ell>0$, see \cite{MAT,Rasmussen2006}. Here, \(K_\nu\) 
denotes the modified Bessel function of the second kind. 
Specifically, property \eqref{eq:MatFourier} holds with
\begin{equation}\label{eq:MaternSymbol}
\widehat{\kernel}_\nu({\bm{\bm\xi}}) 
= \alpha\bigg(1+\frac{\ell^2}{2\nu}\|{\bm {\bm\xi}}\|_2^2\bigg)^{-\nu-d/2},
\end{equation}
where $\alpha$ is a scaling factor depending on 
$\nu$, $\ell$ and $d$, see \cite{MAT}. 
Particularly, the Mat\'ern 
kernels are the reproducing kernels 
of the Sobolev spaces $H^{\nu+d/2}(\mathbb{R}^d)$,
see also \cite{Wendland2004}.

For half integer values of \(\nu\), i.e., for \(\nu=p+1/2\) with 
\(p\in\mathbb{N}_0\), 
the Mat\'ern kernels have an explicit representation given by
\[
\kernel_{p+1/2}(r)=\exp\bigg(\frac {-\sqrt{2\nu}r}{\ell}\bigg)
\frac{p!}{(2p)!}
\sum_{q=0}^p\frac{(p+q)!}{q!(p-q)!}
\bigg(\frac {\sqrt{8\nu}r}{\ell}\bigg)^{p-q}.
\]
The limit case \(\nu\to\infty\) gives rise to the Gaussian kernel
\[
\kernel_\infty(r) = \exp\bigg(\frac{-r^2}{2\ell^2}\bigg).
\]

Our subsequent compression analysis covers the Mat\'{e}rn kernels,
but has considerably wider scope. Indeed, rather large classes of 
pseudodifferential operators will be admissible. As suitable classes 
of such operators are known to define an algebra, properties of 
arithmetic expressions of the underlying kernels, such as off-diagonal 
coefficient decay and matrix compressibility, can directly be inferred.
Equally important, we show that these properties of the operator algebras
are to some extent transferred also to the corresponding finitely represented
structures, i.e., we show the corresponding matrix representation likewise 
are algebras in the compressed format. We refer to Appendix~\ref{sec:pseudos} 
for the details and properties of pseudodifferential operators in this article.

\section{Samplet matrix compression}
\label{sec:compression}
For the readers convenience, we recall in this section
the concept of \emph{samplets} as it has been introduced in \cite{HM21}.
\subsection{Samplets}
Samplets 
are defined based on a sequence of spaces \(\{\Xcal_j\}_{j=0}^J\)
forming a multiresolution analysis, i.e.,
\begin{equation}\label{eq:multiscale}
  \Xcal_0\subset\Xcal_1\subset\cdots\subset\Xcal_J = \Xcal.
\end{equation}
Rather than using a single scale from the
multiresolution analysis \eqref{eq:multiscale}, the idea of 
samplets is to keep track of the increment of information 
between two consecutive levels $j$ and $j+1$. Since we 
have $\Xcal_{j}\subset \Xcal_{j+1}$, we may decompose 
\begin{equation}\label{eq:decomposition}
\Xcal_{j+1} =\Xcal_j\overset{\perp}{\oplus}\Scal_j
\end{equation}
by using the \emph{detail space} $\Scal_j$, where orthogonality
is to be understood with respect to the (discrete)
inner product defined in \eqref{eq:innerp}.

Let ${\bm\Sigma}_j$ be a basis
of the detail space $\Scal_j$ in $\Xcal_j$. 
By choosing a basis \(\bm\Phi_0\) of $\Xcal_0$ and
recursively applying the decomposition \eqref{eq:decomposition},
we see that the set
\[
\mathbf\Sigma_J = {\bm\Phi}_0 \bigcup_{j=0}^J{\bm\Sigma}_j
\]
forms a basis of \(\Xcal_J=\Xcal\), which we call a \emph{samplet basis}. 

In order to employ samplets for the compression of
kernel matrices, it is desirable that the signed measures
$\sigma_{j,k}\in\Xcal_j\subset\Hcal'$
have isotropic convex hulls of supports, and are localized with 
respect to the corresponding discretization level $j$, i.e., 
\begin{equation}\label{eq:locality}
\diam(\supp\sigma_{j,k})\sim 2^{-j/d},
\end{equation}
and that they are stable with respect to the inner product defined
in \eqref{eq:innerp}, i.e.,
\[
\langle \sigma_{j,k},\sigma_{j',k'}\rangle_\Xcal=0
\quad\text{for }(j,k)\neq (j',k').
\]
Furthermore, an essential
ingredient is the 
\emph{vanishing moment condition of order $q+1$}, i.e.,
\begin{equation}\label{eq:vanishingMoments}
 (p,\sigma_{j,k})_\Omega
 = 0\quad \text{for all}\ p\in\Pcal_q(\Omega),
\end{equation}
where \(\Pcal_q(\Omega)\) is the space of all polynomials
with total degree at most \(q\).
We say then that the samplets have \emph{vanishing moments} of
order $q+1$.
%
\begin{remark}
Associated to each samplet
\(\sigma_{j,k} = \sum_{\ell=1}^N\beta_\ell\delta_{{\bm x}_{i_\ell}}\),
we find a uniquely determined function
\[
\hat{\sigma}_{j,k}\isdef
\sum_{\ell=1}^N\beta_\ell \kernel({\bm x}_{i_\ell},\cdot)\in\mathcal{H}_X,
\]
which also exhibits vanishing moments, i.e.,
\[
\langle\hat{\sigma}_{j,k},h\rangle_\Hcal=0
\]
for any \(h\in\Hcal\) which satisfies
\(h|_{\supp\sigma_{j,k}}\in\Pcal_q(\supp\sigma_{j,k})\).
\end{remark}

\subsection{Construction of samplets}
The starting point for the construction of samplets
is the multiresolution analysis \eqref{eq:multiscale}.
Its construction
is based on a hierarchical clustering of the set \(X\).

\begin{definition}\label{def:cluster-tree}
Let $\mathcal{T}=(P,E)$ be a binary tree with vertices $P$ and edges $E$.
We define its set of leaves as
\[
\mathcal{L}(\mathcal{T})\isdef\{\nu\in P\colon\nu~\text{has no sons}\}.
\]
The tree $\mathcal{T}$ is a \emph{cluster tree} for
the set $X=\{{\bm x}_1,\ldots,{\bm x}_N\}$, iff
the set $X$ is the \emph{root} of $\mathcal{T}$ and
all $\nu\in P\setminus\mathcal{L}(\mathcal{T})$
are disjoint unions of their two sons.

The \emph{level} \(j_\nu\) of $\nu\in\mathcal{T}$ is its distance from the root,
i.e., the number of son relations that are required for traveling from
$X$ to $\nu$. The \emph{depth} \(J\) of \(\Tcal\) is the maximum level
of all clusters. We define the set of clusters
on level $j$ as
\[
\mathcal{T}_j\isdef\{\nu\in\mathcal{T}\colon \nu~\text{has level}~j\}.
\]
The cluster tree is \emph{balanced}, iff $|\nu|\sim 2^{J-j_{\nu}}$.
\end{definition}

To bound the diameter of the clusters, we
introduce the \emph{separation radius}
\begin{equation}\label{eq:SepRadius}
q_X\isdef\frac 1 2\min_{i\neq j}\|{\bm x}_i-{\bm x}_j\|_2
\end{equation}
and require \(X\) to be \emph{quasi-uniform}. 

\begin{definition}
\label{def:quasiunif}
The data set $X\subset\Omega$ is \emph{quasi-uniform}
if the fill distance \eqref{eq:FillDistance} is proportional 
to the separation radius \eqref{eq:SepRadius}, i.e., 
there exists a constant $c = c(X,\Omega)\in (0,1)$ 
such that
\[
0<c\le\frac{q_X}{h_{X,\Omega}} \leq c^{-1}.
\]
\end{definition}

Roughly speaking, the points ${\bm x}\in X$ are 
equispaced if $X\subset\Omega$ is quasi-uniform. 
This immediately implies the following result.

\begin{lemma} Let \(\Tcal\) be a cluster tree. For 
\(\nu\in\Tcal\), the \emph{bounding box} $B_{\nu}$ of 
\(\nu\) is the smallest axis-parallel cuboid that contains 
all points of $\nu$. If \(X\subset\Omega\) is quasi-uniform,
then there holds
\[
\frac{|B_\nu|}{|\Omega|} \sim\frac{|B_\nu\cap X|}{N}
\]
with the constant hidden in $\sim$ depending only
on the constant $c(X,\Omega)$ in Definition~\ref{def:quasiunif}.
In particular, we have \(\diam(\nu)\sim 2^{-j_\nu/d}\)
for all clusters \(\nu\in\Tcal\).
\end{lemma}

Samplets with vanishing moments are obtained
recursively by employing a \emph{two-scale} transform between 
basis elements on a cluster $\nu$ of level $j$. To this end, 
we represent \emph{scaling distributions} $\mathbf{\Phi}_{j}^{\nu} 
= \{ \varphi_{j,k}^{\nu} \}$ and \emph{samplets} $\mathbf{\Sigma}_{j}^{\nu} 
= \{ \sigma_{j,k}^{\nu} \}$ as linear combinations of the scaling 
distributions $\mathbf{\Phi}_{j+1}^{\nu}$ of $\nu$'s son clusters. 
This results in the \emph{refinement relation}
\begin{equation}\label{eq:refinementRelation}
   [ \mathbf{\Phi}_{j}^{\nu}, \mathbf{\Sigma}_{j}^{\nu} ] 
 \isdef 
 \mathbf{\Phi}_{j+1}^{\nu}
 {\bm Q}^{\nu}=
 \mathbf{\Phi}_{j+1}^{\nu}
 \big[ {\bm Q}_{j,\Phi}^{\nu},{\bm Q}_{j,\Sigma}^{\nu}\big].
\end{equation}

The transformation matrix \({\bm Q}_{j}^{\nu}\) is computed
from the QR decomposition
\begin{equation}\label{eq:QR} 
  ({\bm M}_{j+1}^{\nu})^\intercal  = {\bm Q}{\bm R}
  \defis\big[{\bm Q}_{j,\Phi}^{\nu} , {\bm Q}_{j,\Sigma}^{\nu}\big]{\bm R}
 \end{equation}
of the \emph{moment matrix}
\[
  {\bm M}_{j+1}^{\nu}\isdef
  \begin{bmatrix}({\bm x}^{\bm 0},\varphi_{j+1,1})_\Omega&\cdots&
  ({\bm x}^{\bm 0},\varphi_{j+1,|\nu|})_\Omega\\
  \vdots & & \vdots\\
  ({\bm x}^{\bm\alpha},\varphi_{j+1,1})_\Omega&\cdots&
  ({\bm x}^{\bm\alpha},\varphi_{j+1,|\nu|})_\Omega
  \end{bmatrix}=
  [({\bm x}^{\bm\alpha},\mathbf{\Phi}_{j+1}^{\nu})_\Omega]_{|\bm\alpha|\le q}
\in\Rbb^{m_q\times|\nu|}
\]
with
\begin{equation*}
m_q\isdef\sum_{\ell=0}^q{\ell+d-1\choose d-1}\leq(q+1)^d
\end{equation*}
being the dimension of \(\Pcal_q(\Omega)\).
There holds
\begin{equation}\label{eq:vanishingMomentsQR}
  \begin{aligned}
  \big[{\bm M}_{j,\Phi}^{\nu}, {\bm M}_{j,\Sigma}^{\nu}\big]
  &= \left[({\bm x}^{\bm\alpha},[\mathbf{\Phi}_{j}^{\nu},
  	\mathbf{\Sigma}_{j}^{\nu}])_\Omega\right]_{|\bm\alpha|\le q}
  = \left[({\bm x}^{\bm\alpha},\mathbf{\Phi}_{j+1}^{\nu}[{\bm Q}_{j,\Phi}^{\nu}
  , {\bm Q}_{j,\Sigma}^{\nu}])_\Omega
  	\right]_{|\bm\alpha|\le q} \\
  &= {\bm M}_{j+1}^{\nu} [{\bm Q}_{j,\Phi}^{\nu} , {\bm Q}_{j,\Sigma}^{\nu} ]
  = {\bm R}^\intercal.
  \end{aligned}
\end{equation}
As ${\bm R}^\intercal$ is a lower triangular matrix, the first $k-1$ 
entries in its $k$-th column are zero. This corresponds to 
$(k-1)$ vanishing moments for the $k$-th function generated 
by the transformation $[{\bm Q}_{j,\Phi}^{\nu} , {\bm Q}_{j,\Sigma}^{\nu} ]$. 
By defining the first $m_{q}$ functions as scaling distributions and 
the remaining as samplets, we obtain samplets with vanishing 
moments at least up to order $q+1$.

For leaf clusters, we define the scaling distributions by
the Dirac measures at the points \({\bm x}_i\), i.e.,
$\mathbf{\Phi}_J^{\nu}\isdef\{ \delta_{{\bm x}_i} : {\bm x}_i\in\nu \}$,
to make up for the lack of son clusters that could provide scaling distributions.
The scaling distributions of all clusters on a specific level $j$ 
then generate the spaces
\begin{equation}\label{eq:Vspaces}
	\Xcal_{j}\isdef \spn\{ \varphi_{j,k}^{\nu} : 
	k\in \Delta_j^\nu,\ \nu \in\Tcal_{j} \},
\end{equation}
while the samplets span the detail spaces
\begin{equation}\label{eq:Wspaces}
	\Scal_{j}\isdef
	\spn\{ \sigma_{j,k}^{\nu} : k\in \nabla_j^\nu,\ \nu \in \Tcal_{j} \} =
	\Xcal_{j+1}\overset{\perp}{\ominus}\Xcal_j.
\end{equation}
Combining the scaling distributions of the root cluster with all 
clusters' samplets amounts to the final samplet basis
\begin{equation}\label{eq:Wbasis}
  \mathbf{\Sigma}_{N}\isdef\mathbf{\Phi}_{0}^{X} 
  	\cup \bigcup_{\nu \in\Tcal} \mathbf{\Sigma}_{j_{\nu}}^{\nu}.
\end{equation}
A visualization of a scaling distribution and
different samplets on a spiral data set is found in Figure~\ref{fig:samplets}.

\begin{figure}[hbt]
\begin{center}
\includegraphics[trim=40 30 40 30,clip,width=0.32\textwidth]{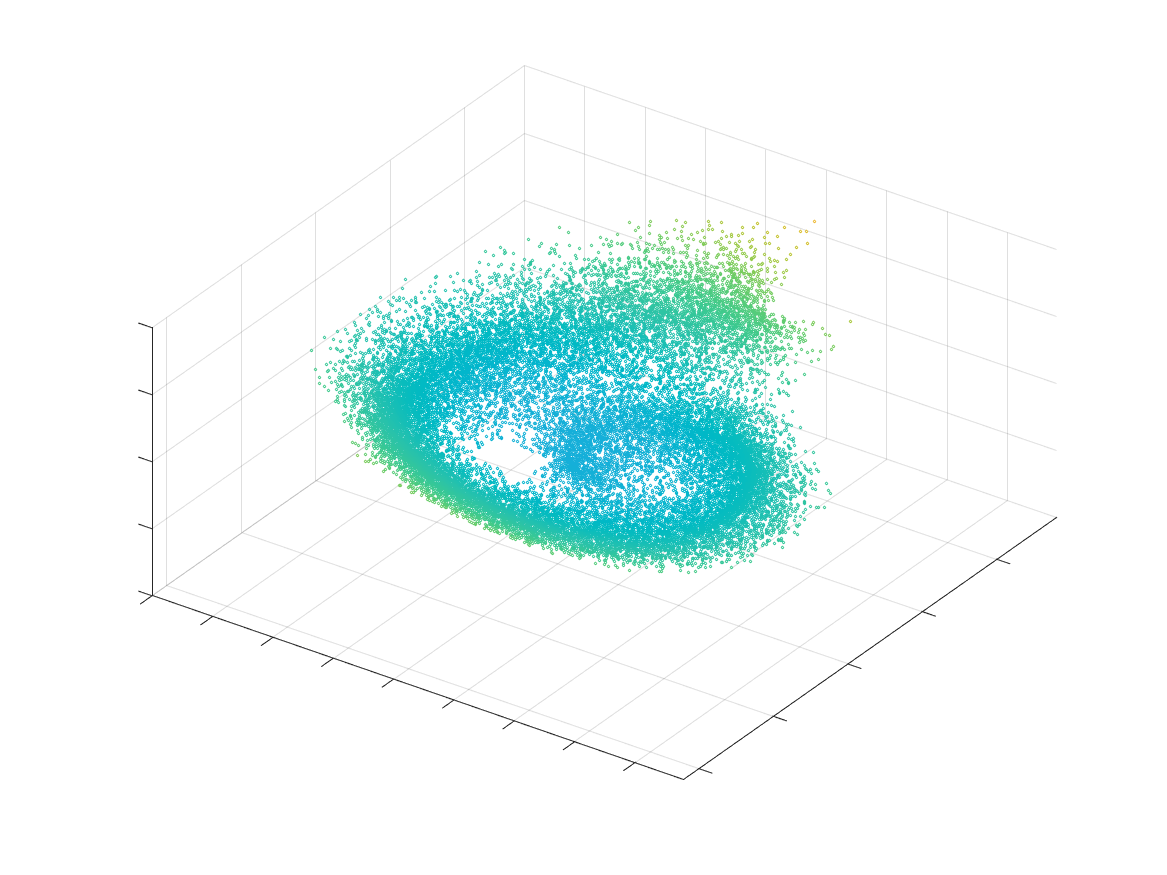}
\includegraphics[trim=40 30 40 30,clip,width=0.32\textwidth]{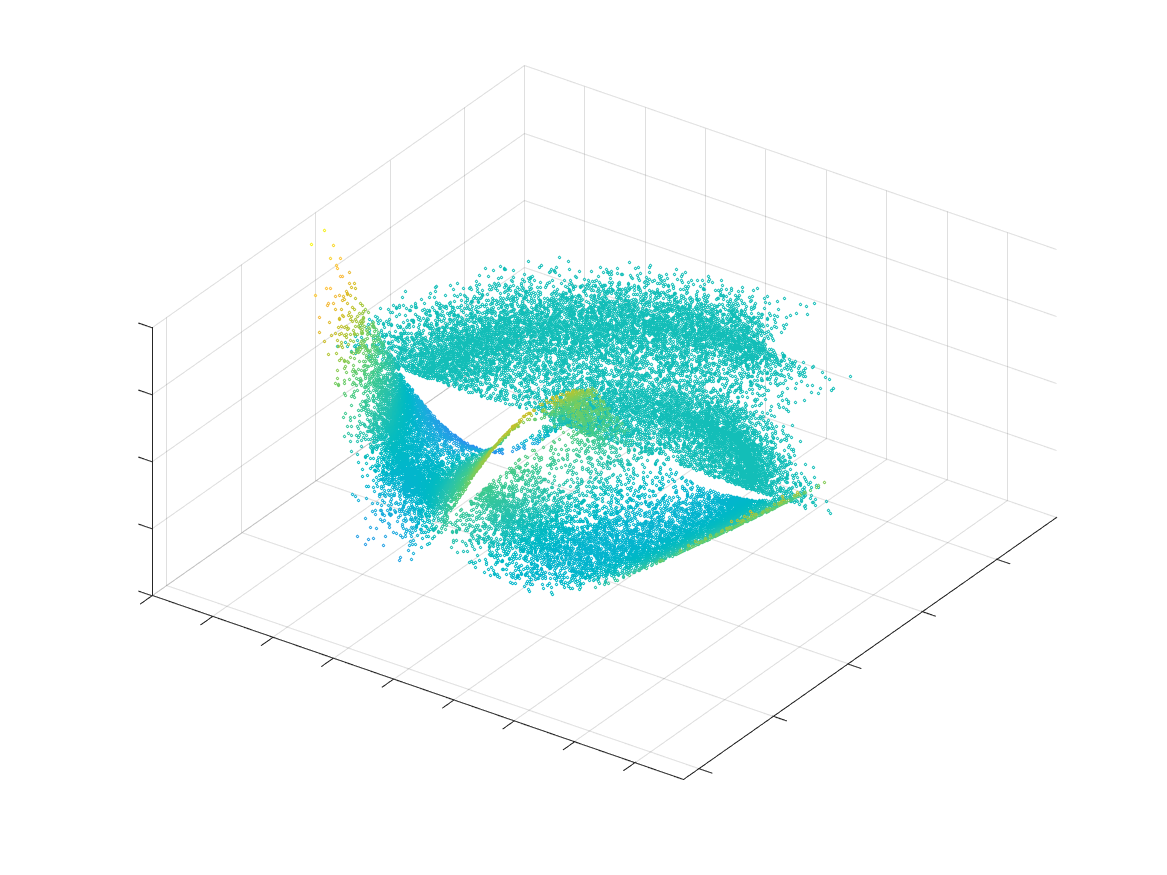}
\includegraphics[trim=40 30 40 30,clip,width=0.32\textwidth]{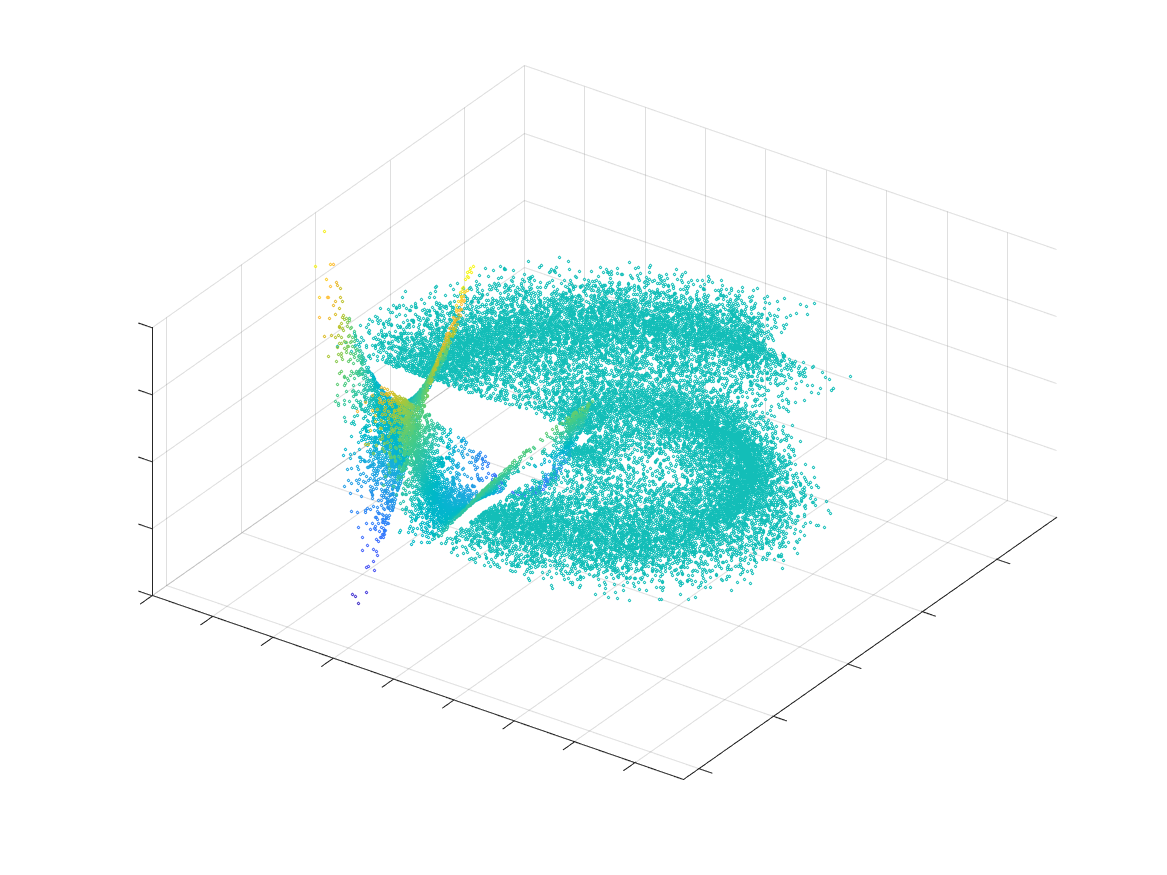}
\end{center}
\caption{\label{fig:samplets}
A scaling distribution on the coarsest scale (\emph{left}) and
samplets on level 2 and 3 (\emph{second from the left to right}).}
\end{figure}

By construction, samplets satisfy the following 
properties, which are collected from
\cite[Theorem 3.6, Lemma 3.9, Theorem 5.4]{HM21}.

\begin{theorem}\label{theo:waveletProperties}
The spaces $\Xcal_{j}$ defined in equation \eqref{eq:Vspaces} 
form the desired multiresolution analysis \eqref{eq:multiscale},
where the corresponding complement spaces $\Scal_{j}$ from 
\eqref{eq:Wspaces} satisfy 
\[
\Xcal_{j+1}=\Xcal_j\overset{\perp}{\oplus}\Scal_{j}\quad
\text{for all}\quad j=0,1,\ldots, J-1. 
\]
The associated samplet basis 
$\mathbf{\Sigma}_{N}$ defined in \eqref{eq:Wbasis} constitutes 
an orthonormal basis of $\Xcal$ and we have:
\begin{enumerate}
\item 
The number of all samplets on level $j$ behaves like $2^j$.
\item 
The samplets have vanishing moments of order $q+1$, 
i.e., there holds \eqref{eq:vanishingMoments}.
\item 
Each samplet is supported in a specific cluster $\nu$. 
If the points in $X$ are quasi-uniform, then the 
diameter of the cluster satisfies $\diam(\nu)\sim 
2^{-j_\nu/d}$ and there holds \eqref{eq:locality}.
\item 
The coefficient vector 
${\bm\omega}_{j,k}=\big[\omega_{j,k,i}\big]_i$ 
of the samplet $\sigma_{j,k}$ on the cluster $\nu$ fulfills
\begin{equation*}
  \|{\bm\omega}_{j,k}\|_{1}\le\sqrt{|\nu|}.
\end{equation*}
\item Let $f\in C^{q+1}(\Omega)$. Then, there holds 
for a samplet $\sigma_{j,k}$ supported on the cluster $\nu$ that
\begin{equation*}
 |(f,\sigma_{j,k})_\Omega|\le \bigg(\frac{d}{2}\bigg)^{q+1}
  	\frac{\diam(\nu)^{q+1}}{(q+1)!}\|f\|_{C^{q+1}(\Omega)}
	\|{\bm\omega}_{j,k}\|_{1}.
\end{equation*}
\end{enumerate}
\end{theorem}

\begin{remark}
Each samplet is a linear combination of the Dirac 
measures supported at the points in $X$. The related 
coefficient vectors ${\bm\omega}_{j,k}$ in
\begin{equation}\label{eq:coefficientVectorsOfWavelets}
  \sigma_{j,k} = \sum_{i=1}^{N}
  \omega_{j,k,i} \delta_{{\bm x}_i}
\end{equation}
are pairwise orthonormal with respect to the inner
product \eqref{eq:innerp}. 
The \emph{dual samplet} in \(\Hcal_X\) is given by
\[
\tilde{\sigma}_{j,k}=\sum_{i=1}^N\tilde{\omega}_{j,k,i}\kernel({\bm x}_i,\cdot),
\quad\text{where}\quad\tilde{\bm\omega}_{j,k}\isdef{\bm K}^{-1}{\bm\omega}_{j,k},
\]
as there holds
\begin{align*}
\langle\tilde{\sigma}_{j,k},\hat{\sigma}_{j'k'}\rangle_{\Hcal}=
(\tilde{\sigma}_{j,k},\sigma_{j'k'})_\Omega
&=\sum_{i,i'=1}^N\tilde{\omega}_{j,k,i}{\omega}_{j',k',i'}
\big(\kernel({\bm x}_i,\cdot),\delta_{{\bm x}_{i'}}\big)_\Omega\\
&=\tilde{\bm\omega}_{j,k}^\intercal{\bm K}{\bm\omega}_{j',k'}
=\delta_{(j,k),(j',k')}.
\end{align*}
\end{remark}

\subsection{Matrix compression}
For the compression of the kernel matrix \({\bm K}\) from
\eqref{eq:KernelMatrix}, with samplets of vanishing moment order
$q+1$, with integer $q\geq 0$,
we suppose that kernel \(\kernel\) is ``$q+1$-\emph{asymptotically smooth}''.
This is to say that
there are constants $c_{\kernel,\bm\alpha,\bm\beta}>0$
such that for all ${\bm x},{\bm y}\in \Omega$ 
with ${\bm x}\ne {\bm y}$ there holds
\begin{equation}\label{eq:kernel_estimate}
  \bigg|\frac{\partial^{|\bm\alpha|+|\bm\beta|}}
  	{\partial{\bm x}^{\bm\alpha} \partial{\bm y}^{\bm\beta}}
	\kernel({\bm x},{\bm y})\bigg|
  		\le c_{\kernel,\bm\alpha,\bm\beta}
		\|{\bm x}-{\bm y}\|_2^{-(|\bm\alpha|+|\bm\beta|)} 
\quad\text{for all } |\bm\alpha|, |\bm\beta| \leq q+1 
\;.
\end{equation}

Note that such an estimate can only be valid for continuous 
kernels as considered here, but not for singular kernels. 
However,
we observe in passing that this condition is considerably weaker than
the notion of asymptotic smoothness of kernels in ${\mathcal H}$-matrix 
theory, cp.\ \cite{Hackbusch}. 
The condition there would correspond to 
infinite differentiability in \eqref{eq:kernel_estimate} 
with analytic estimates on the constants $c_{\kernel,\bm\alpha,\bm\beta}$.

Due to \eqref{eq:kernel_estimate}, we have in 
accordance with \cite[Lemma 5.3]{HM21} the decay estimate
\begin{equation}\label{eq:kernel_decay}
  (\kernel,\sigma_{j,k}\otimes\sigma_{j',k'})_{\Omega\times\Omega}\le
  	c_{\kernel,q}\frac{\diam(\nu)^{q+1}\diam(\nu')^{q+1}}
		{\dist(\nu_{j,k},\nu_{j',k'})^{2(q+1)}}
		\|{\bm\omega}_{j,k}\|_{1}\|{\bm\omega}_{j',k'}\|_{1}
\end{equation}
for two samplets $\sigma_{j,k}$ and $\sigma_{j',k'}$,
with the vanishing moment property of order $q+1$ and
supported on the
clusters \(\nu\) and \(\nu'\) such that $\dist(\nu,\nu') > 0$.

Estimate \eqref{eq:kernel_decay} holds for a wide range of 
kernels that obey the so-called \emph{Calder\'on-Zygmund estimates}.
It immediately results in 
the following compression strategy for kernel matrices in
samplet representation, cp.\ \cite[Theorem 5.4]{HM21},
which is well-known in the context of wavelet compression
of operator equations see, e.g., \cite{Meyer}.
\begin{theorem}[$S$-compression]\label{thm:compression}
Set all coefficients of the kernel matrix
\[
{\bm K}^\Sigma\isdef\big[(\kernel,\sigma_{j,k}
\otimes\sigma_{j',k'})_{\Omega\times\Omega}
\big]_{j,j',k,k'}
\]
to zero which satisfy the \emph{$\eta$-admissibility condition}
\begin{equation}\label{eq:cutoff}
   \dist(\nu,\nu')\ge\eta\max\{\diam(\nu),\diam(\nu')\},\quad\eta>0,
\end{equation}
where \(\nu\) is the cluster supporting \(\sigma_{j,k}\) and
\(\nu'\) is the cluster supporting \(\sigma_{j',k'}\), respectively.

Then, the resulting \emph{S-compressed} matrix \({\bm K}^\eta\)
satisfies
\[
  \big\|{\bm K}^\Sigma-{\bm K}^\eta\big\|_F
	\le c {\eta^{-2(q+1)}} 
		N\sqrt{\log(N)}.
\]
for some constant \(c>0\) dependent on the polynomial
degree \(q\) and the kernel \(\kappa\).
\end{theorem}

\begin{remark}
We remark that Theorem~\ref{thm:compression} 
uses the Frobenius norm for measuring the error rather than the 
operator norm, as it gives control on each matrix coefficient. 
Estimates with respect to the operator norm would be similar.

The $\eta$-admissibility condition \eqref{eq:cutoff} 
appears reminiscent to the one used for hierarchical matrices, 
compare, e.g., \cite{Boe10} and the references there.
However, in the
present context, the clusters \(\nu\) and \(\nu'\) may also be
located on different levels, i.e., \(j_\nu\neq j_{\nu'}\) in general.
As a consequence, 
the resulting block cluster tree is the cartesian product
\(\Tcal\times\Tcal\) rather than the level-wise cartesian product
considered in the context of hierarchical matrices.

The error bounds for $S$-compression hold for kernel functions $\kappa$ with
finite differentiability (especially, with derivatives of order $q+1$, cp.\
\cite[Lemma 5.3]{HM21}), 
as 
opposed to the usual requirement of asymptotic smoothness which appears in 
the error analysis of the $\Hcal$-format, see \cite{Hackbusch}
and the references therein.
\end{remark}

For point sets $X = \{{\bm x}_i\}_{i=1}^N$ 
that are quasi-uniform in the sense of 
Definition~\ref{def:quasiunif}, 
there holds
\[
 \frac{1}{N^2}\big\|{\bm K}^\Sigma\big\|_F^2 = 
 \frac{1}{N^2}\sum_{i=1}^N\sum_{j=1}^N |\kernel({\bm x}_i,{\bm x}_j)|^2 
 \sim \int_\Omega\int_\Omega|\kernel({\bm x},{\bm y})|^2\d{\bm x}\d{\bm y},
\]
i.e., $\big\|{\bm K}^\Sigma\big\|_F\sim N$. 
Thus, we can refine the above result, see also \cite[Corollary 5.5]{HM21}. 
\begin{corollary}
In case of quasi-uniform points ${\bm x}_i\in X$, the $S$-compressed 
matrix ${\bm K}^\eta$ has only $\mathcal{O}(N\log N)$ 
nonzero coefficients, while it satisfies the error estimate
\begin{equation}\label{eq:relerr}
  \frac{\big\|{\bm K}^\Sigma-{\bm K}^\eta\big\|_F}
  {\big\|{\bm K}^\Sigma\big\|_F} \le c \eta^{-2(q+1)} \sqrt{\log N}.
\end{equation}
\end{corollary}

In \cite{HM21}, an algorithm has been proposed which
computes the compressed matrix ${\bm K}^\eta$ 
in work and memory \(\mathcal{O}(N\log N)\). 
The key ingredient to achieve
this is the use of an interpolation-based fast multipole method
and $\Hcal^2$-matrix techniques \cite{Boe10,GR,AHK14}.
\section{Samplet matrix algebra}\label{sec:arithmetic}
\subsection{Addition and multiplication}
\label{sec:AddMlt}
To bound the cost for the addition of 
two compressed kernel matrices represented with
respect to the same cluster tree, it is sufficient
to assume that the points in $X$ are quasi-uniform.
Then it is straightforward to see that the cost
for adding such matrices is \(\mathcal{O}({N\log N})\).
The multiplication of two compressed matrices, in turn,
is motivated by
concatenation $\Ccal=\Acal\circ\Bcal$ of the two
pseudodifferential 
operators $\Acal$ and $\Bcal$. 
In suitable algebras, the product
\(\Ccal\) is again a pseudodifferential operator
and, hence, compressible.
The respective kernel 
$\kernel_{\Ccal}(\cdot,\cdot)$
is given by
\begin{equation}\label{eq:kernel_C}
 \kernel_{\Ccal}({\bm x},{\bm y}) 
= \int_\Omega\kernel_{\Acal}({\bm x},{\bm z})
 \kernel_{\Bcal}({\bm z},{\bm y})\d {\bm z}.
\end{equation}

Since
\(\Omega\subset\Rbb^d\) is bounded by assumption, 
we may without loss of generality assume \(\Omega\subset[0,1)^d\).
Then, if the distribution of the
data points in $X=\{\bm{x}_i\}_{i=1}^N\subset\Omega$ satisfies the
stronger assumption of being \emph{asymptotically uniform modulo one}, 
then there holds
\begin{equation}\label{eq:unifmod1}
\lim_{N\to\infty}\frac{|\Omega|}{N}\sum_{i=1}^N(f,\delta_{{\bm x}_i})_\Omega=
\int_\Omega f({\bm x})\d{\bm x}
\end{equation}
for every Riemann integrable function \(f\colon\Omega\to\Rbb\), cp.\ \cite{LP10}.
Hence, we may interpret the matrix product as a discrete version of the
convolution \eqref{eq:kernel_C}. In view of \eqref{eq:unifmod1}, we conclude
\begin{equation}\label{eq:integral}
 \bigg|\kernel_{\Ccal}({\bm x},{\bm y})
-\frac{|\Omega|}{N}\sum_{k=1}^N
  \kernel_{\Acal}({\bm x},{\bm x}_k) 
   \kernel_{\Bcal}({\bm x}_k,{\bm y})\bigg|\to 0\ \text{as $N\to\infty$}.
\end{equation}
Consequently, the product of two kernel matrices 
\[
\bm{K}_{\Acal} = [\kernel_{\Acal}(\bm{x}_i,\bm{x}_j)]_{i,j=1}^N,\quad
\bm{K}_{\Bcal} = [\kernel_{\Bcal}(\bm{x}_i,\bm{x}_j)]_{i,j=1}^N
\] 
yields an $S$-compressible matrix \(\bm{K}_{\Acal}\cdot
\bm{K}_{\Bcal}\in\Rbb^{N\times N}\). 

\begin{theorem}
Let $X=\{\bm{x}_i\}_{i=1}^N \subset\Omega$
be asymptotically distributed
uniformly modulo one, cp.\ \eqref{eq:unifmod1}, and 
denote by $\bm{K}_{\Ccal}$ the corresponding kernel matrix
\[
\bm{K}_{\Ccal} = \frac{N}{|\Omega|}
[\kernel_{\Ccal}(\bm{x}_i,\bm{x}_j)]_{i,j=1}^N
\]
with $\kernel_{\Ccal}(\cdot,\cdot)$ from \eqref{eq:kernel_C}.
Then, there holds 
\[
  \frac{\|\bm{K}_{\Ccal}-\bm{K}_{\Acal}\bm{K}_{\Bcal}\|_F}{
  \|\bm{K}_{\Ccal}\|_F} \to 0
\ \text{as $N\to\infty$}.
\]
\end{theorem}

\begin{proof}
On the one hand, we conclude from \eqref{eq:integral} that, as $N\to \infty$,
\begin{align*}
  \|\bm{K}_{\Ccal}-\bm{K}_{\Acal}\bm{K}_{\Bcal}\|_F^2
  	&= \sum_{i,j=1}^N
		\bigg[\frac{N}{|\Omega|}\kernel_{\Ccal}({\bm x}_i,{\bm x}_j) 
		- \sum_{k=1}^N 
	\kernel_{\Acal}({\bm x}_i,{\bm z}_k)\kernel_{\Bcal}({\bm z}_k,{\bm x}_j)
		\bigg]^2\\
 &\sim N^4 \int_\Omega\int_\Omega	\bigg[\kernel_{\Ccal}({\bm x},{\bm y})
 	-\frac{|\Omega|}{N}\sum_{k=1}^N 
	\kernel_{\Acal}({\bm x},{\bm x}_k)\kernel_{\Bcal}({\bm x}_k,{\bm y})
	\bigg]^2\d\bm{x}\d\bm{y}\\	
		 &= o(N^4).
\end{align*}
On the other hand, we find likewise
\[
  \|\bm{K}_{\Ccal}\|_F^2\sim\int_\Omega\int_\Omega N^2
   \kernel_{\Ccal}({\bm x},{\bm y})^2\d\bm{x}\d\bm{y}
	\sim N^4.
\]
This implies the assertion.
\end{proof}

\begin{remark}
We mention that the consistency bound 
in the preceding theorem is rather crude.
Under provision of stronger kernel-function regularity, 
corresponding higher convergence rates can be achieved, 
given that $X$ satisfies
appropriate higher-order quasi-Monte Carlo designs, 
see, e.g., \cite{DKP22} and the references there.
\end{remark}

Let \(\bm{K}_{\Acal}^\eta,\bm{K}_{\Bcal}^\eta,\bm{K}_{\Ccal}^\eta\)
be compressed with respect to the same compression pattern.
We assume for given \(\varepsilon(\eta)>0\)
that \(\eta\) in \eqref{eq:relerr} is chosen
such that
\[
  \big\|{\bm K}^\Sigma-{\bm K}^\eta\big\|_F
   \le\varepsilon(\eta){\big\|{\bm K}^\Sigma\big\|_F},\quad
   \text{for }{\bm K}\in\{\bm{K}_{\Acal},\bm{K}_{\Bcal},\bm{K}_{\Ccal}\}.
\]
Then, a repeated application of the triangle inequality yields
\begin{align*}
&\|\bm{K}_{\Ccal}^\eta-\bm{K}_{\Acal}^\eta\bm{K}_{\Bcal}^\eta\|_F\\
&\quad\le\|\bm{K}_{\Ccal}^\Sigma-\bm{K}_{\Ccal}^\eta\|_F + 
\|\bm{K}_{\Acal}^\Sigma\|_F\|\bm{K}_{\Bcal}^\Sigma-\bm{K}_{\Bcal}^\eta\|_F
+\|\bm{K}_{\Bcal}^\eta\|_F\|\bm{K}_{\Acal}^\Sigma-\bm{K}_{\Acal}^\eta\|_F\\
&\quad\leq\varepsilon(\eta)\big(\|\bm{K}_{\Ccal}\|_F
+\|\bm{K}_{\Acal}\|_F\|\bm{K}_{\Bcal}\|_F
+\big(1+\varepsilon(\eta)\big)\|\bm{K}_{\Acal}\|_F\|\bm{K}_{\Bcal}\|_F\big)\\
&\quad\lesssim \varepsilon(\eta)\big(\|\bm{K}_{\Ccal}\|_F
+\|\bm{K}_{\Acal}\|_F\|\bm{K}_{\Bcal}\|_F\big).
\end{align*}
This means that we only need to compute 
$\mathcal{O}(N\log N)$ matrix entries to determine an approximate 
version $(\bm{K}_{\Acal}^\eta\bm{K}_{\Bcal}^\eta)^\eta$
of the product $\bm{K}_{\Acal}^\eta\cdot\bm{K}_{\Bcal}^\eta$. 
We like to stress that this formatted matrix multiplication is exact 
on the given compression pattern. The next theorem gives a cost bound on the
matrix multiplication.
\begin{theorem}\label{thm:product}
Consider two kernel matrices 
\[
\bm{K}_{\Acal}^\eta=[a_{(j,k),(j',k')}],
\quad 
\bm{K}_{\Bcal}^\eta=[b_{(j,k),(j',k')}]\in\Rbb^{N\times N}
\] 
in samplet representation which are $S$-compressed with respect to the 
compression
pattern induced by the $\eta$-admissibility condition \eqref{eq:cutoff}. 

Then, computing with respect to the same compression pattern 
the matrix 
$\bm{K}_{\Ccal}^\eta=[c_{(j,k),(j',k')}]\in\Rbb^{N\times N}$, 
where the nonzero entries are given by the discrete inner product 
\begin{equation}\label{eq:product}
  c_{(j,k),(j',k')} = \sum_{\ell=0}^J\sum_{m\in\nabla_\ell} 
  a_{(j,k),(\ell,m)} b_{(\ell,m),(j',k')},
\end{equation}
is of cost $\mathcal{O}(N\log^2 N)$.
\end{theorem}

\begin{proof}
To estimate the cost of the matrix multiplication,
we shall make use of the compression rule \eqref{eq:cutoff}.
We assume for all clusters that $\diam(\nu)\sim 2^{-j/d}$ if $\nu$ is 
on level $j$. Thus, the samplet $\sigma_{j,k}$ has approximately 
the diameter $2^{-j/d}$ and, therefore, only $\mathcal{O}(2^{\ell-j})$ 
samplets $\sigma_{\ell,m}$ of diameter $\sim 2^{-\ell/d}$ are found in its 
nearfield if $\ell\ge j$ while only $\mathcal{O}(1)$ are found if $\ell<j$. 
For fixed level $0\le\ell\le J$ in \eqref{eq:product}, we thus have at 
most $\mathcal{O}(\max\{2^{\ell-\max\{j,j'\}},1\})$ nonzero products
to evaluate per coefficient $c_{(j,k),(j',k')}$. 
We assume without loss of 
generality that $j\ge j'$ and sum over $\ell$, which yields the cost 
$\mathcal{O}(\max\{2^{J-j},j\})$. Per target block matrix ${\bm C}_{j,j'} 
= [c_{(j,k),(j',k')}]_{j,j'}$, we have $\mathcal{O}(2^{\max\{j,j'\}}) = 
\mathcal{O}(2^j)$ nonzero coefficients. Hence, the cost for
computing the desired target block is $\mathcal{O}(2^j
\max\{2^{J-j},j\})$. We shall next sum over $j$ and $j'$
\begin{align*}
\sum_{j=0}^J \sum_{j'=0}^j \mathcal{O}(2^j\max\{2^{J-j},j\})
&= \sum_{j=0}^J \sum_{j'=0}^j \mathcal{O}(\max\{N,j2^j\})\\
&= \sum_{j=0}^J \mathcal{O}(j\max\{N,j 2^j\}) \\
&= \mathcal{O}(N\log^2 N).
\end{align*}
\end{proof}

\subsection{Sparse selected inversion}\label{subs:selinv}
\label{sec:SPSI}
Having addition and multiplication of kernel matrices
at our disposal, we consider the matrix inversion next.
To this end, observe that the inverse \({\Acal}^{-1}\) of a 
pseudodifferential operator \({\Acal}\) from a suitable algebra 
of pseudodifferential operators, provided that it exists, is again 
a pseudodifferential operator, see Section \ref{sec:pseudos}. 
However, if ${\Acal}$ is a pseudodifferential operator of negative 
order as in the present RKHS case, the operator
\({\Acal}^{-1}\) is of positive order and hence gives rise to
a singular kernel which does  not satisfy the condition
\eqref{eq:kernel_estimate}. Even so, in the regime of
kernel matrices we are rather interested in inverting regularized 
pseudodifferential operators, i.e., ${\Acal}+\mu{I}$,
where \(I\) denotes the identity.
For such operators, we have the following lemma.

\begin{lemma}
Let ${\Acal}$ be a pseudodifferential operator of order $s \le 0$
with symmetric and positive semidefinite kernel function. 

Then, for any $\mu>0$,
the inverse of ${\Acal}+\mu{I}$ can be decomposed into
$\frac{1}{\mu}{I}-{\Bcal}$ with
\begin{equation}\label{eq:perturbation}
  {\Bcal} = \frac{1}{\mu}({\Acal}+\mu{I})^{-1}{\Acal}.
\end{equation}
Especially, ${\Bcal}$ is also a pseudodifferential operator of order 
$s$, which admits a symmetric and positive semidefinite kernel function.
\end{lemma}

\begin{proof}
In view of \eqref{eq:perturbation}, we infer that
\[
 ({\Acal}+\mu{I})\bigg(\frac{1}{\mu}{I}-{\Bcal}\bigg) 
 	= \frac{1}{\mu}{\Acal} +I - ({\Acal}+\mu{I}){\Bcal}
	= I + \frac{1}{\mu}{\Acal}- \frac{1}{\mu}{\Acal} = I.
\]
Therefore, $\frac{1}{\mu}{I}-{\Bcal}$ is the inverse operator
to ${\Acal}+\mu{I}$. Since ${\Acal}+\mu{I}$ is of order 0, $({\Acal}
+\mu{I})^{-1}$ is of order 0, too, and thus $({\Acal}+\mu{I})^{-1}
{\Acal}$ is of the same order as ${\Acal}$. Finally, the symmetry 
and nonnegativity of $\Bcal$ follows from the symmetry and 
nonnegativity of $\Acal$. 
\end{proof}

As a consequence of this lemma, the inverse \(({\bm K}_{\Acal}+\mu{\bm I})^{-1}
\in\Rbb^{N\times N}\) of the associated kernel matrix 
\({\bm K}_{\Acal}+\mu {\bm I}
\in\Rbb^{N\times N}\) is $S$-compressible with respect 
to the same compression pattern as \({\bm K}_{\Acal}\). In \cite{HM}, 
strong numerical evidence was presented that a sparse Cholesky
factorization of a compressed kernel matrix can efficiently be computed by
means of nested dissection, cf.\ \cite{Geo73}. 
This suggests the computation of the inverse 
\(({\bm K}_{\Acal}+\mu{\bm I})^{-1}\)
in samplet basis on the compression pattern of \({\bm K}_{\Acal}\) 
by means of selective inversion~\cite{selinv} of a sparse matrix. 
The approach is outlined below.

Assume that ${\bm A}\in\Rbb^{N\times N}$ is symmetric and positive definite.
There are two steps in the inversion algorithm. The first stage involves 
factorizing the input matrix $\bm{A}$ into $\bm{A}=\bm{LDL}^\intercal$.
The $\bm{L}$ and $\bm{D}$ matrices are used in the second phase to compute 
the selected components of $\bm{A}^{-1}$. The first step will be referred to 
as factorization in the following and the second step as selected inversion.
To explain the second step, let ${\bm A}$ be partitioned according to
\[
{\bm A}=
\begin{bmatrix}
{\bm A}_{11} & {\bm A}_{12}\\ {\bm A}_{12}^\intercal &
{\bm A}_{22}
\end{bmatrix}.
\]
In particular, the diagonal blocks ${\bm A}_{ii}$ are also symmetric
and positive definite.
The selected inversion is based on the identity
\begin{equation}\label{eq:blockInv}
{\bm A}^{-1}
=
\begin{bmatrix}
{\bm A}_{11}^{-1}
+{\bm C}{\bm S}^{-1}{\bm C}^\intercal & {\bm C}{\bm S}^{-1}\\ 
{\bm S}^{-1}{\bm C}^\intercal &
{\bm S}^{-1}
\end{bmatrix},
\end{equation}
where \({\bm S}\isdef {\bm A}_{22}+{\bm A}_{12}^\intercal{\bm C}\)
is the Schur complement with \({\bm C}\isdef-{\bm A}_{11}^{-1}{\bm A}_{12}\).
For sparse matrices, this block algorithm can efficiently be realized based on
the observation that for the computation of the entries of \({\bm A}^{-1}\)
on the pattern of \({\bm L}\) only the entries on the pattern of \({\bm L}\)
are required, as it is well known from the sparse matrix literature,
cp.\ \cite{selinv,DR83,GL81}. 
We note that the pattern of ${\bm A}$ 
is particularly contained in the pattern of \({\bm L}\).
\subsection{Algorithmic aspects}
\label{sec:AlgAsp}
A block selected inversion algorithm has at least two advantages: 
Because $\bm{A}$ is sparse, blocks can be specified in terms of supernodes
\cite{selinv}. This allows us to use level-3 \texttt{BLAS} to construct an efficient 
implementation by leveraging memory hierarchy in current microprocessors. 
A supernode is a group of nodes with the same nonzero structure below the 
diagonal in their respective columns (of their $\bm{L}$ factor). The supernodal 
approach for sparse symmetric factorization represents the factor $\bm{L}$ as 
a set of supernodes, each of which consists of a contiguous set of $\bm{L}$ 
columns with identical nonzero patterns, and each supernode is stored as 
a dense submatrix to take advantage of level-3 \texttt{BLAS} calculations.

Taking these consideration
as a starting point, it is natural to employ the selected inversion approach
presented in~\cite{RINLA:2021} and available in~\cite{PARDISO:2022}
in order to directly compute the entries on the pattern of the inverse matrix.
For the particular implementation of the selected inversion, we rely on
\texttt{Pardiso}. 
For larger kernel matrices, which cannot be indexed
by \(32bit\) integers due to the comparatively large number of 
non-zero entries, we combine the selected inversion with a divide and conquer
approach based on the identity \eqref{eq:blockInv}.
The inversion of the \({\bm A}_{11}\) block and of the Schur complement
\({\bm S}\) are performed with \texttt{Pardiso} (exploiting symmetry), 
while the other arithmetic operations, i.e., addition and multiplication, 
are performed in a formatted way, compare\ Theorem~\ref{thm:product}.

\subsection{Matrix functions}
\label{sec:MatFun}
%
Based on the $S$-formatted multiplication and inversion 
of operators represented in samplet basis, 
certain holomorphic functions of an $S$-compressed operator
also admit $S$-formatted approximations with, essentially,
corresponding approximation accuracies.

To illustrate this,
we recall the method in \cite{Hale2008}.
This approach employs the \emph{contour integral representation}
\begin{equation}\label{eq:Cntr}
  f({\bm A}) = \frac{1}{2\pi i} \int_\Gamma f(z) (z{\bm I}-{\bm A})^{-1}\d z,
\end{equation}
where $\Gamma$ is a closed contour being contained in the 
analyticity region of $f$ and winding once around the spectrum
$\sigma({\bm A})$ in counterclockwise direction. As is well-known, 
analytic functions $f$ of elliptic, self-adjoint pseudodifferential 
operators yield again pseudodifferential operators in the same 
algebra, see, e.g., \cite[Chap.XII.1]{Taylor81}. Hence, ${\bm B} 
\isdef f({\bm A})$ is $S$-compressible provided that $f$ is analytic. 
Especially, the $S$-compressed representation 
$\big(f({\bm A}^\eta)\big)^{\eta}$ satisfies
\begin{equation}\label{eq:generalMatFunc}
\begin{aligned}
\big\|{\bm B}^\Sigma-\big(f({\bm A}^{\eta})\big)^{\eta}\big\|_F
&\leq\|{\bm B}^\Sigma-{\bm B}^{\eta}\|_F
+\big\|\big(f({\bm A}^\Sigma)-f({\bm A}^{\eta})\big)^{\eta}\big\|_F\\
&\leq \varepsilon\|{\bm B}\|_F + L\|{\bm A}^\Sigma-{\bm A}^{\eta}\|_F\\
&\leq\varepsilon\big(\|{\bm B}\|_F+L\|{\bm A}\|_F\big).
\end{aligned}
\end{equation}
Herein, \(L\) denotes the Lipschitz constant of the function \(f\).
In other words, estimate \eqref{eq:generalMatFunc} implies 
that the error of the approximation of the $S$-formatted matrix 
function $\big(f({\bm A}^{\eta})\big)^{\eta}$ is 
rigorously controlled by the sum of the input error 
$\|{\bm A}^\Sigma-{\bm A}^{\eta}\|_F$ and the 
compression error for the exact output $\|{\bm B}^\Sigma
-{\bm B}^{\eta}\|_F$. The latter is under control
if the underlying pseudodifferential operator is of order
$s < -d$ since then the kernel is continuous and satisfies
\eqref{eq:kernel_estimate}. In the other cases, some
analysis is needed to control this error (see below).

For the numerical approximation of the contour integral  \eqref{eq:Cntr},
one has to apply an appropriate quadrature formula. 
Exemplarily, we consider the matrix square root,
i.e., $f(z) = \sqrt{z}$ for ${\rm Re} z > 0$.
This occurs for example in the efficient path simulation of Gaussian 
processes in spatial statistics. 
We shall here apply, see~\cite[Eq.~(4.4) and comments below]{Hale2008},
the approximation 
\begin{equation}\label{eq:ContInt}
{\bm A}^{-1/2}\approx\frac{2 E \sqrt{\underline{c}}}{\pi K} 
\sum_{k=1}^K\frac{\operatorname{dn}
\left(t_k | 1-\varkappa_{\bm A}\right)}
{\operatorname{cn}^2\left(t_k | 1 - \varkappa_{\bm A}\right)}  
\left({\bm A} + w_k^2{\bm I}\right)^{-1},
\quad {\bm A}^{1/2} = {\bm A}\cdot {\bm A}^{-1/2}.
\end{equation}
Here, $\operatorname{sn}, \operatorname{cn}$ and 
$\operatorname{dn}$ are the Jacobian elliptic 
functions~\cite[Chapter~16]{Abramowitz1964}, 
$E$ is the complete elliptic integral of the second kind
associated with the parameter $\varkappa_{\bm A} := 
\underline{c}/\overline{c}$ \cite[Chapter~17]{Abramowitz1964}, 
and, for $k\in\{1,\ldots, K \}$, 
\[
w_k\isdef\sqrt{\underline{c}}\,\frac{
 	\operatorname{sn}\left(t_k | 1 - \varkappa_{\bm A}\right)}
	{\operatorname{cn}\left(t_k | 1 - \varkappa_{\bm A}\right)}
\quad 
\text{and} 
\quad 
t_k\isdef\frac{E}{K}\big(k-\tfrac{1}{2}\big).
\]
The quadrature approximation \eqref{eq:ContInt}
of the contour integral \eqref{eq:Cntr}
for the matrix square root is known to converge 
root-exponentially (e.g.~\cite[Lemma 3.4]{BP15})
in the number $K$ of quadrature nodes in \eqref{eq:ContInt} 
of the contour integral. Hence, approximate representations 
to algebraic with respect to $N$
consistency orders can be achieved with 
$K\sim|\varepsilon(\eta)|$, resulting in overall log-linear complexity 
of numerical realization of \eqref{eq:ContInt} in $S$-format.
We also remark that the quadrature shifts $w_k^2$ in the inversions
which occur in \eqref{eq:ContInt}
act as regularizing ``nuggets'' of a possibly ill-conditioned 
${\bm A}$.
The input parameters $0< \underline{c} < \overline{c}$ 
shall provide bounds to the spectrum of ${\bm A}$, i.e., 
$\underline{c}\approx\lambda_{\min}({\bm A})$ and 
$\overline{c}\approx\lambda_{\max}({\bm A})$. Note that we also
assume here that ${\bm A}$ is symmetric and positive definite.
Moreover, we should mention that, except for the quadrature error, 
\eqref{eq:ContInt} computes the square root $({\bm A}^\eta)^{-1/2}$
of the compressed input ${\bm A}^\eta$ in an exact way on the 
compression pattern when we use the selected inversion algorithm
from Subsection~\ref{subs:selinv}.

That $({\bm A}^\eta)^{-1/2}$ is indeed $S$-compressible 
is a consequence of the following lemma.

\begin{lemma}
Let ${\Acal}$ be a pseudodifferential operator of order $s \leq 0$
with symmetric and positive semidefinite kernel function. 
Then, for any $\mu>0$,
the inverse square root of ${\Acal}+\mu{I}$ can be 
written as
$\frac{1}{\sqrt{\mu}}{I}-{\Bcal}$ 
with ${\Bcal}$ 
being also a pseudodifferential operator of order $s$, which admits a 
symmetric and positive semidefinite kernel function.
\end{lemma}

\begin{proof}
Straightforward calculation shows that the ansatz
\begin{equation}\label{eq:ansatz}
(\Acal+\mu I)^{-1/2} = \frac{1}{\sqrt{\mu}} I-\Bcal
\end{equation}
is equivalent to 
\[
(\Acal+\mu I)\bigg(\frac{1}{\mu} I-\frac{2}{\sqrt{\mu}}\Bcal+\Bcal^2\bigg) = I.
\]
Thus, 
\[
\Bcal\bigg(\frac{2}{\sqrt{\mu}}I-\Bcal\bigg) 
= \frac{1}{\mu}(\Acal+\mu I)^{-1}\Acal,
\]
which in view of \eqref{eq:ansatz} is equivalent to
\[
\Bcal\bigg(\frac{1}{\sqrt{\mu}}I+(\Acal+\mu I)^{-1/2}\bigg) 
= \frac{1}{\mu}\Acal(\Acal+\mu I)^{-1}.
\]
As both, $\frac{1}{\sqrt{\mu}}I+(\Acal+\mu I)^{-1/2}$ and $(\Acal+\mu I)^{-1}$,
are pseudodifferential of order 0, $\Bcal$ must have the
same order as $\Acal$. 
\end{proof}

An alternative to the contour integral for computing the
matrix exponential of a (possibly singular) matrix ${\bm A}$
is given by the direct evaluation of the power series
\[
\exp({\bm A})=\sum_{k=0}^\infty\frac{1}{k!}{\bm A}^k.
\]
As we show in the numerical results, this series converges
very fast for the present matrices under consideration which 
stem from reproducing kernels, since they correspond to 
compact operators.

\section{Numerical results}\label{sec:results}
The computations in this section have been performed on
a single node with two Intel Xeon E5-2650 v3 @2.30GHz
CPUs and up to 512GB of main memory\footnote{The full
specifications can be found on
https://www.euler.usi.ch/en/research/resources.}.
To achieve consistent timings, all computations have
been carried out using 16 cores.
The samplet compression is implemented in \texttt{C++11}
and relies on the
\texttt{Eigen} template
library\footnote{\texttt{https://eigen.tuxfamily.org/}}
for linear algebra operations. Moreover, the selected
inversion is performed by \texttt{Pardiso}.
Throughout this section, we employ samplets with
\(q+1=4\) vanishing moments. The parameter for the
admissibility condition \eqref{eq:cutoff} is set to
\(\eta=1.25\). Together with the \emph{a priori pattern},
which is obtained by neglecting admissible blocks, we
also consider an a posteriori compression by setting
all matrix entries smaller than \(\tau=10^{-5}/N\) to zero
resulting in the \emph{a posteriori pattern}.
\subsection{$\boldsymbol{S}$-formatted matrix multiplication}
To benchmark the multiplication, we consider uniformly
distributed random points on the unit hypercube \([0,1]^d\).
As kernel, we consider the exponential kernel (which is
the Mat\'ern kernel with smoothness parameter $\nu=1/2$
and correlation length $\ell=1$)
\begin{equation}\label{eq:expKernel2}
\kernel({\bm x},{\bm y})=\frac{1}{N}e^{-\|{\bm x}-{\bm y}\|_2}
\end{equation}
Note that we impose the scaling $1/N$ of the kernel function
in order fix the largest eigenvalue of the kernel matrix as its
trace stays uniformly bounded.

We compute the matrix product \({\bm K}^\eta
\cdot\tilde{\bm K}^\eta\), where 
\(\tilde{\bm K}^\eta\) is obtained
from \({\bm K}^\eta\) by relatively perturbing 
each nonzero entry by 10\% additive noise, which is uniformly 
distributed in \([0,1]\). This way, we rule out symmetry effects 
as \(\tilde{\bm K}^\eta\) will not be 
symmetric in general.

\begin{figure}[htb]
\begin{center}
\begin{tikzpicture}
\begin{loglogaxis}[width=0.475\textwidth,grid=both,
 ymin= 5e-1, ymax = 1.4e5, xmin = 5e3, xmax =1e6, title={multiplication time},
    legend style={legend pos=south east,font=\tiny},
    ytick = {1e-1, 1e0, 1e1,1e2,1e3,1e4,1e5},
     ylabel={\small wall time}, xlabel ={\small $N$}]
\addplot[line width=0.7pt,color=blue,mark=square]
table[select coords between index={0}{6},x=npts,y=mult_tim]{./data/output_Multiplication_FINAL_2D.txt};
\addlegendentry{$d=2$};
\addplot[line width=0.7pt,color=red,mark=triangle]
table[select coords between index={0}{6},x=npts,y=mult_tim]{./data/output_Multiplication_FINAL_3D.txt};
\addlegendentry{$d=3$};
\addplot[line width=0.7pt,color=black,dashed]
table[select coords between index={0}{6},x=npts,y expr={0.01 * x}]{./data/output_Inversion_FINAL_apriori_2D.txt};
\addplot[line width=0.7pt,color=black,dashed]
table[select coords between index={0}{6},x=npts,y expr={0.01 * x * ln(x) / ln(5000)}]{./data/output_Inversion_FINAL_apriori_2D.txt};
\addplot[line width=0.7pt,color=black,dashed]
table[select coords between index={0}{6},x=npts,y expr={0.01 * x * ln(x)^2 / ln(5000)^2}]{./data/output_Inversion_FINAL_apriori_2D.txt};
\addplot[line width=0.7pt,color=black,dashed]
table[select coords between index={0}{6},x=npts,y expr={0.01 * x * ln(x)^3 / ln(5000)^3}]{./data/output_Inversion_FINAL_apriori_2D.txt};
\addlegendentry{$\mathcal{O}(N\log^\alpha N)$};
\end{loglogaxis}
\end{tikzpicture}
\begin{tikzpicture}
\begin{loglogaxis}[width=0.475\textwidth,grid=both,
 ymin= 1e-8, ymax = 2e-6, xmin = 5e3, xmax =1e6, title={multiplication error},
    legend style={legend pos=north east,font=\tiny},
    ytick = {1e-8, 1e-7, 1e-6,1e-5},
     ylabel={\small $e_F$}, xlabel ={\small $N$}]
\addplot[line width=0.7pt,color=blue,mark=square]
table[select coords between index={0}{6},x=npts,y=mult_err]{./data/output_Multiplication_FINAL_2D.txt};
\addlegendentry{$d=2$};
\addplot[line width=0.7pt,color=red,mark=triangle]
table[select coords between index={0}{6},x=npts,y=mult_err]{./data/output_Multiplication_FINAL_3D.txt};
\addlegendentry{$d=3$};
\end{loglogaxis}
\end{tikzpicture}
\end{center}
\caption{\label{fig:Mult}Computation times for matrix multiplication (left) and multiplication
errors (right).}
\end{figure}
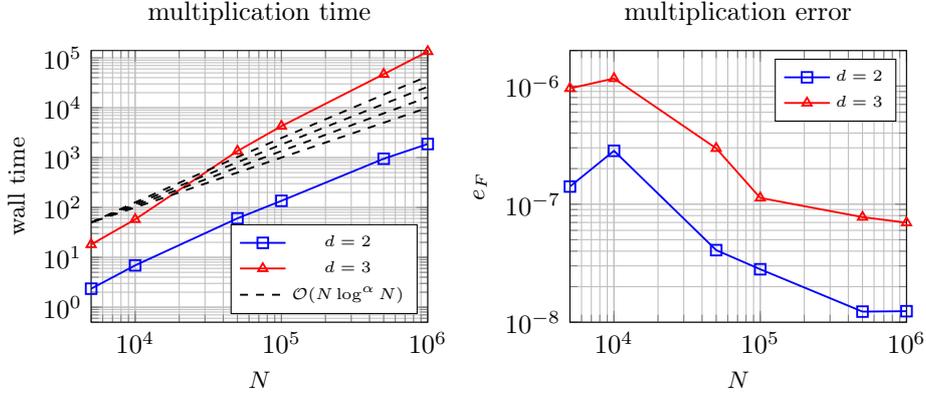

To measure the multiplication error, we consider the estimator
\[
e_F({\bm A})\isdef\frac{\|{\bm A}{\bm X}\|_F}{\|{\bm X}\|_F},
\]
where \({\bm X}\in\Rbb^{N\times 10}\) is a random matrix with uniformly
distributed independent entries.
The left hand side of Figure~\ref{fig:Mult} shows the computation time for a single
multiplication. The dashed lines correspond to the 
asymptotic rates \(\Ocal(N\log^\alpha N)\) for \(\alpha=0,1,2,3\). 
It can be seen that the multiplication time for \(d=2\)
perfectly reflects the expected essentially linear behavior. Though the
graph is steeper for \(d=3\), we expect it to flatten further for larger \(N\).
The right hand side of the plot shows the multiplication error 
\(e_F({\bm K}^\eta\cdot\tilde{\bm K}^\eta-
{\bm K}^\eta\boxdot\tilde{\bm K}^\eta)\),
where the formatted multiplication $\boxdot$ is performed on the a posteriori pattern.
Taking into account that the compression errors for \({\bm K}^\eta\)
are approximately \(5.6\cdot10^{-6}\) for \(d=2\) and \(1.6\cdot10^{-5}\) for \(d=3\),
the obtained matrix product can be considered to be very accurate.
\subsection{$\boldsymbol{S}$-formatted matrix inversion}
\label{sec:MatInv}
In order to assess the numerical performance of the matrix inversion,
we again consider uniformly distributed random points on the unit hypercube \([0,1]^d\).
Since the separation radius \(q_X\)
ranges between \(4.7\cdot10^{-5}\) ($N=5000$) and \(2.8\cdot10^{-7}\) 
($N=1\, 0000\, 000$) for \(d=2\) and 
\(3.8\cdot10^{-4}\) ($N=5000$) and \(3.2\cdot10^{-5}\) ($N=1\, 0000\, 000$)
for \(d=3\), we do not expect that
\({\bm K}^\eta\) to be invertible. Therefore, we rather consider the
regularized version \({\bm K}^\eta +\mu{\bm I}\)
for a ridge parameter \(\mu>0\).

As our theoretical results suggest that the inverse has the same a priori pattern as 
the matrix itself, we first consider the inversion on the a priori pattern for \(d=2\).

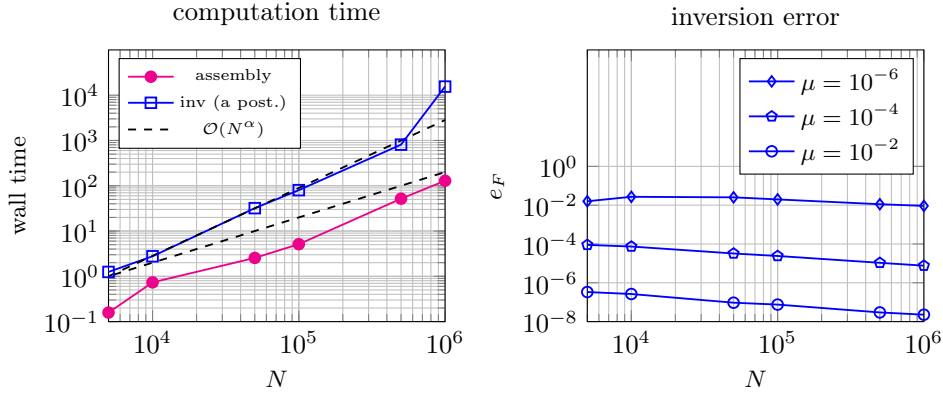
\begin{figure}
\begin{center}
\begin{tikzpicture}
\begin{loglogaxis}[width=0.475\textwidth,grid=both,
 ymin= 1e-1, ymax = 1e5, xmin = 5e3, xmax =1e6, title={computation time},
    legend style={legend pos=north west,font=\tiny},
    ytick = {1e-1, 1e0, 1e1,1e2,1e3,1e4},
     ylabel={\small wall time}, xlabel ={\small $N$}]
\addplot[line width=0.7pt,color=magenta,mark=*]
table[select coords between index={0}{6},x=npts,y=comp_time]{./data/output_Inversion_FINAL_apriori_2D.txt};
\addlegendentry{assembly};
\addplot[line width=0.7pt,color=blue,mark=square]
table[select coords between index={0}{6},x=npts,y=pardiso_t]{./data/output_Inversion_FINAL_apriori_2D.txt};
\addlegendentry{inv (a post.)};
\addplot[line width=0.7pt,color=black,dashed]
table[select coords between index={0}{6},x=npts,y expr={0.0002 * x}]{./data/output_Inversion_FINAL_apriori_2D.txt};
\addplot[line width=0.7pt,color=black,dashed]
table[select coords between index={0}{6},x=npts,y expr={0.0002 * x^1.5 / 5000^0.5}]{./data/output_Inversion_FINAL_apriori_2D.txt};
\addlegendentry{$\mathcal{O}(N^\alpha)$};
\end{loglogaxis}
\end{tikzpicture}
\begin{tikzpicture}
\begin{loglogaxis}[width=0.475\textwidth,grid=both, title={inversion error},
 ymin= 1e-8, ymax = 1e6, xmin = 5e3, xmax =1e6,
    legend style={legend pos=north east,font=\small},
     ytick = {1e-8,1e-6,1e-4,1e-2,1e0},
     ylabel={\small $e_F$}, xlabel ={\small $N$}]
\addplot[line width=0.7pt,color=blue,mark=diamond]
table[select coords between index={0}{6},x=npts,y=inv_err]{./data/output_Inversion_FINAL_apriori_2D.txt};
\addlegendentry{$\mu=10^{-6}$}
\addplot[line width=0.7pt,color=blue,mark=pentagon]
table[select coords between index={7}{13},x=npts,y=inv_err]{./data/output_Inversion_FINAL_apriori_2D.txt};
\addlegendentry{$\mu=10^{-4}$}
\addplot[line width=0.7pt,color=blue,mark=o]
table[select coords between index={14}{20},x=npts,y=inv_err]{./data/output_Inversion_FINAL_apriori_2D.txt};
\addlegendentry{$\mu=10^{-2}$}
\end{loglogaxis}
\end{tikzpicture}
\caption{\label{fig:inversion2Dap}Results for \(d=2\).
Left panel: Computation times for compressed matrix
assembly and selected inversion on the a priori pattern.
Dashed lines indicate linear ($\alpha=1$) and super-linear ($\alpha=1.5$) scaling, respectively.
Right panel: Inversion errors for ridge parameters \(\mu=10^{-6},10^{-4},10^{-2}\).}
\end{center}
\end{figure}
\setlength{\fboxsep}{0.2pt} 

The left hand side of Figure~\ref{fig:inversion2Dap} shows the computation times for the
inverse matrix employing \texttt{Pardiso}. 
The dashed line shows the asymptotic rates \(\Ocal(N^\alpha)\) for \(\alpha=1,1.5\).
For \(N=1\,000\,000\), due to the large amount
of non-zero entries, we use the block inversion with one subdivision. This explains the
bump in the computation time due to the formatted matrix multiplication. Besides this,
\texttt{Pardiso} perfectly exhibits the expected rate of \(N^{1.5}\).
The right hand side of the plot shows the error
\(e_F\big(({\bm K}^\eta
+\mu{\bm I})^{\fbox{\text{\tiny$-\!1$}}}({\bm K}^\eta +\mu{\bm I})-{\bm I}\big)\)
for the ridge parameters \(\mu=10^{-6},10^{-4},10^{-2}\), where
\({}^{\fbox{\text{\tiny$-\!1$}}}\) denotes the selected in version on the pattern of
\({\bm K}^\eta\).
As expected, the error reduces significantly with increasing ridge parameter.

As the a priori pattern typically exhibits significantly less entries, we also investigate
the inversion on the a posteriori pattern. The corresponding results are 
shown in
Figure~\ref{fig:inversion2D}
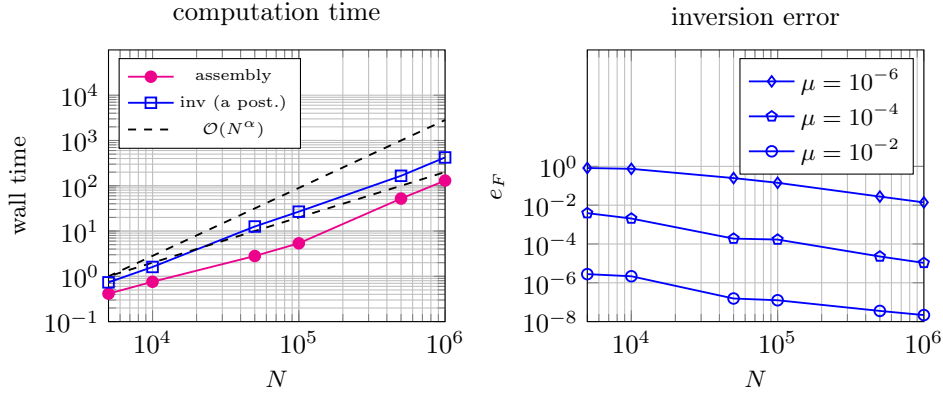
\begin{figure}
\begin{center}
\begin{tikzpicture}
\begin{loglogaxis}[width=0.475\textwidth,grid=both,
 ymin= 1e-1, ymax = 1e5, xmin = 5e3, xmax =1e6, title={computation time},
    legend style={legend pos=north west,font=\tiny},
    ytick = {1e-1, 1e0, 1e1,1e2,1e3,1e4},
     ylabel={\small wall time}, xlabel ={\small $N$}]
\addplot[line width=0.7pt,color=magenta,mark=*]
table[select coords between index={0}{6},x=npts,y=comp_time]{./data/output_Inversion_FINAL_2D.txt};
\addlegendentry{assembly};
\addplot[line width=0.7pt,color=blue,mark=square]
table[select coords between index={0}{6},x=npts,y=pardiso_t]{./data/output_Inversion_FINAL_2D.txt};
\addlegendentry{inv (a post.)};
\addplot[line width=0.7pt,color=black,dashed]
table[select coords between index={0}{6},x=npts,y expr={0.0002 * x}]{./data/output_Inversion_FINAL_apriori_2D.txt};
\addplot[line width=0.7pt,color=black,dashed]
table[select coords between index={0}{6},x=npts,y expr={0.0002 * x^1.5 / 5000^0.5}]{./data/output_Inversion_FINAL_apriori_2D.txt};
\addlegendentry{$\mathcal{O}(N^\alpha)$};
\end{loglogaxis}
\end{tikzpicture}
\begin{tikzpicture}
\begin{loglogaxis}[width=0.475\textwidth,grid=both, title={inversion error},
 ymin= 1e-8, ymax = 1e6, xmin = 5e3, xmax =1e6,
    legend style={legend pos=north east,font=\small},
     ytick = {1e-8,1e-6,1e-4,1e-2,1e0},
     ylabel={\small $e_F$}, xlabel ={\small $N$}]
\addplot[line width=0.7pt,color=blue,mark=diamond]
table[select coords between index={0}{6},x=npts,y=inv_err]{./data/output_Inversion_FINAL_2D.txt};
\addlegendentry{$\mu=10^{-6}$}
\addplot[line width=0.7pt,color=blue,mark=pentagon]
table[select coords between index={8}{14},x=npts,y=inv_err]{./data/output_Inversion_FINAL_2D.txt};
\addlegendentry{$\mu=10^{-4}$}
\addplot[line width=0.7pt,color=blue,mark=o]
table[select coords between index={16}{22},x=npts,y=inv_err]{./data/output_Inversion_FINAL_2D.txt};
\addlegendentry{$\mu=10^{-2}$}
\end{loglogaxis}
\end{tikzpicture}\caption{\label{fig:inversion2D}Results for \(d=2\).
Left panel: Computation times for compressed matrix
assembly and selected inversion on the a posteriori pattern.
	Dashed lines indicate linear ($\alpha=1$) and super-linear ($\alpha=1.5$) scaling, respectively.
	Right panel: Inversion errors for ridge parameters \(\mu=10^{-6},10^{-4},10^{-2}\).}
	\end{center}
	\end{figure}

As can be seen on the left hand side of the figure, the selected inversion
now even exhibits a linear behavior, which is explained by the fixed threshold \(\tau\),
resulting in successively less entries for increasing \(N\). On the other hand,
the errors for the different ridge parameters, depicted on the right hand side of the
same figure, asymptotically exhibit the same behavior as in the a priori case.

\begin{figure}
\begin{center}
\begin{tikzpicture}
\begin{loglogaxis}[width=0.47\textwidth,grid=both,
 ymin= 3e-1, ymax = 1e5, xmin = 5e3, xmax =1e6, title={computation time},
    legend style={legend pos=north west,font=\tiny},
    ytick = {1e-1, 1e0, 1e1,1e2,1e3,1e4},
     ylabel={\small wall time}, xlabel ={\small $N$}]
\addplot[line width=0.7pt,color=magenta,mark=*]
table[select coords between index={0}{6},x=npts,y=comp_time]{./data/output_Inversion_FINAL_3D.txt};
\addlegendentry{assembly};
\addplot[line width=0.7pt,color=red,mark=triangle]
table[select coords between index={0}{6},x=npts,y=pardiso_t]{./data/output_Inversion_FINAL_3D.txt};
\addlegendentry{inv (apost)};

\addplot[line width=0.7pt,color=black,dashed]
table[select coords between index={0}{6},x=npts,y expr={0.001 * x}]{./data/output_Inversion_FINAL_apriori_2D.txt};
\addplot[line width=0.7pt,color=black,dashed]
table[select coords between index={0}{6},x=npts,y expr={0.001 *x^2 / 5000}]{./data/output_Inversion_FINAL_apriori_2D.txt};
\addlegendentry{$\mathcal{O}(N)$};
\end{loglogaxis}
\end{tikzpicture}
\begin{tikzpicture}
\begin{loglogaxis}[width=0.47\textwidth,grid=both, title={error},
 ymin= 1e-8, ymax = 1e6, xmin = 5e3, xmax =1e6,
    legend style={legend pos=north east,font=\small},
     ytick = {1e-8,1e-6,1e-4,1e-2,1e0},
     ylabel={\small wall time}, xlabel ={\small $N$}]
\addplot[line width=0.7pt,color=red,mark=diamond]
table[select coords between index={0}{6},x=npts,y=inv_err]{./data/output_Inversion_FINAL_3D.txt};
\addlegendentry{$\mu=10^{-6}$}
\addplot[line width=0.7pt,color=red,mark=pentagon]
table[select coords between index={7}{13},x=npts,y=inv_err]{./data/output_Inversion_FINAL_3D.txt};
\addlegendentry{$\mu=10^{-4}$}
\addplot[line width=0.7pt,color=red,mark=o]
table[select coords between index={14}{20},x=npts,y=inv_err]{./data/output_Inversion_FINAL_3D.txt};
\addlegendentry{$\mu=10^{-2}$}
\end{loglogaxis}
\end{tikzpicture}
\caption{\label{fig:inversion3D}Results for \(d=3\).
Left panel: Computation times for compressed matrix
assembly and selected inversion on the a posteriori pattern.
Dashed lines indicate linear ($\alpha=1$) and quadratic ($\alpha=2$) scaling, respectively.}
Right panel: Inversion errors for ridge parameters \(\mu=10^{-6},10^{-4},10^{-2}\).
\end{center}
\end{figure}
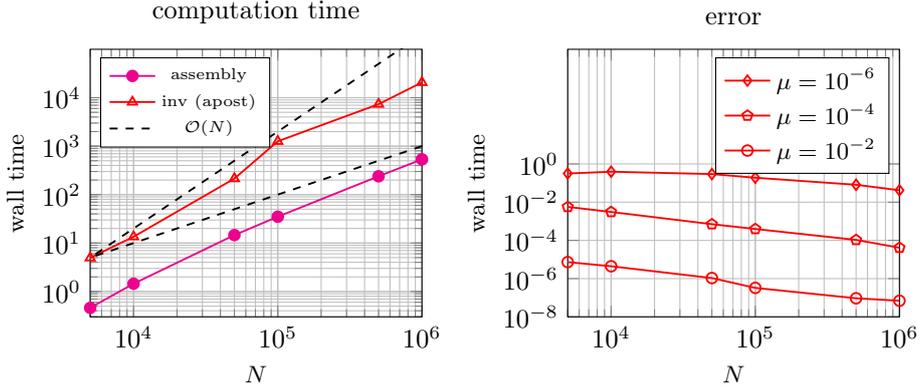

Motivated by the results for \(d=2\), we consider only the inversion on the a posteriori
pattern for \(d=3\). The corresponding results are shown in Figure~\ref{fig:inversion3D}.
On the left hand side of the figure, again the computation times are shown.
The dashed lines show the asymptotic rates \(\Ocal(N^\alpha)\) for \(\alpha=1,2\).
Until \(N=100\,000\), the expected quadratic rate is perfectly matched.
Due to the large number of non-zeros in the case \(d=3\), we have employed the 
block inversion with three recursion steps for \(N>100\,000\), resulting in the
peculiar linear behavior for the respective values in the graph.
The errors depicted on the right hand side show a behavior similar to the
case \(d=2\), with a slightly reduced decay.
%
\subsection{$\boldsymbol{S}$-formatted matrix functions}
%
We compute the matrix square root \({\bm A}^{1/2}\) and the matrix
exponential \(\exp({\bm A})\) for the exponential kernel
\[
\kernel({\bm x},{\bm y})=\frac{1}{N}e^{-2\|{\bm x}-{\bm y}\|_2}
\]
This time, the data points are randomly subsampled from from a 3D
scan of the head of Michelangelo’s David
(The scan is provided by the Statens Museum for Kunst under
the Creative Commons CC0 license), cp.\ Figure~\ref{fig:David}.
\begin{figure}
\begin{center}
\includegraphics[clip, trim = 860 300 860 300,scale=0.085]{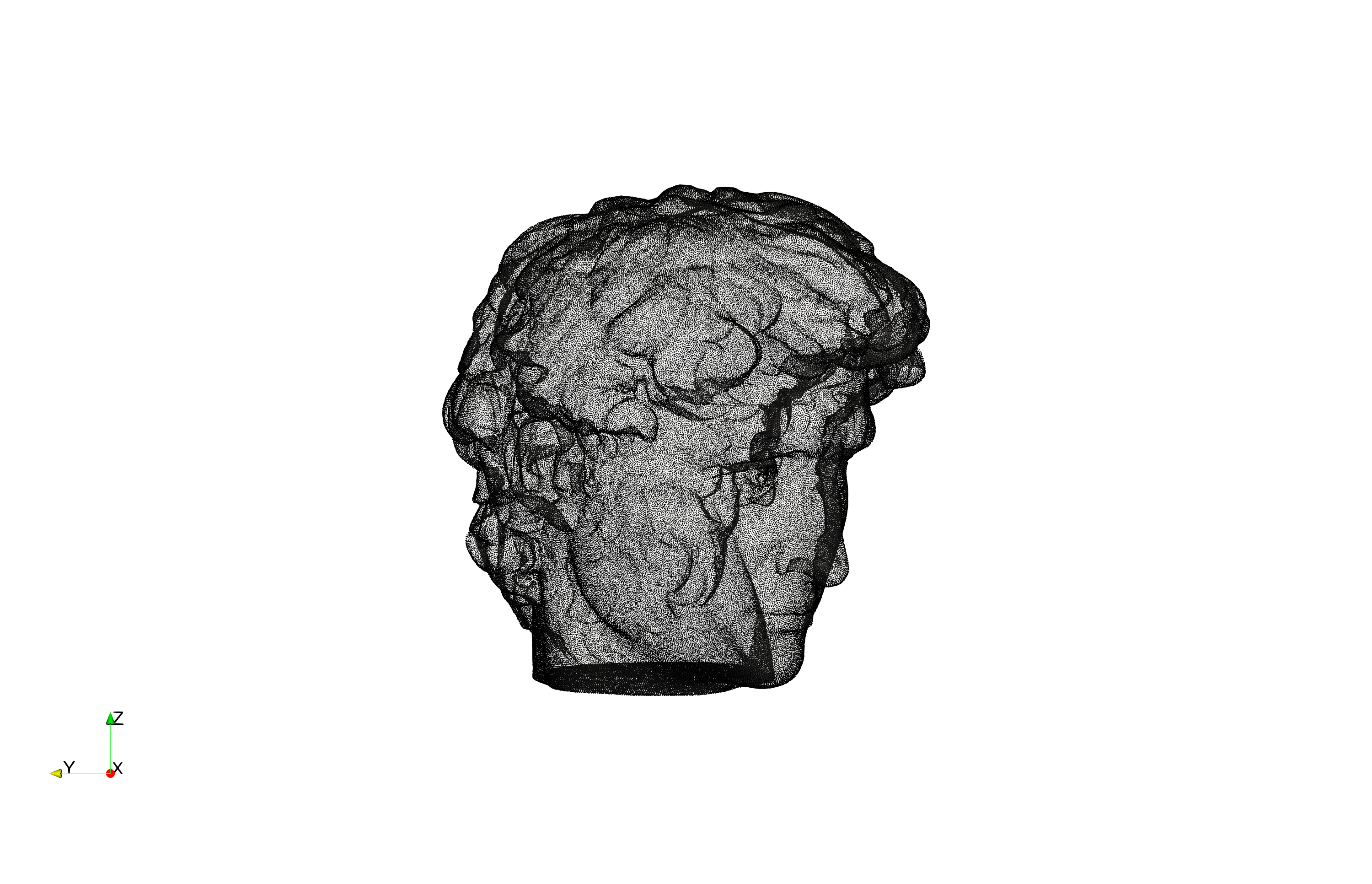}
\caption{\label{fig:David}Data points from a 3D scan of the head
of Michelangelo's David. The scan is provided by the Statens 
Museum for Kunst under the Creative Commons CC0 license.}
\end{center}
\end{figure}
The bounding box of the bunny is \([-0.52,0.42]\times[-0.47,0.46]\times[-0.18,0.78]\).
All other parameters are set as in the examples before. Moreover,
we set the ridge parameter to \(\mu=10^{-4}\). The smallest eigenvalue
is estimated by the ridge parameter, while the largest eigenvalue is upper bounded
by \(1\).
For the contour integral method for the computation of the matrix square root,
we found stagnation in the error for \(K\geq 7\) quadrature points. The
corresponding errors 
\(e_F\big(({\bm K}^\eta+\mu{\bm I})^{\fbox{\text{\tiny$1/2$}}}
({\bm K}^\eta+\mu{\bm I})^{\fbox{\text{\tiny$1/2$}}}-
({\bm K}^\eta+\mu{\bm I})\big)\)
for different values of \(N\) are tabulated in Table~\ref{tab:contour}
\begin{table}[htb]
\begin{center}
\begin{tabular}{|c|r|r|r|r|r|}\hline
$N$ & $5\, 000$ & $10\, 000$ & $50\, 000$ & $100\,000$ & $234\,553$\\\hline
$e_F$ & $1.43\cdot 10^{-3}$& $7.68\cdot 10^{-4}$& $3.90\cdot 10^{-4}$& $2.82\cdot 10^{-4}$& $3.59\cdot 10^{-4}$\\\hline
\end{tabular}
\caption{\label{tab:contour}Errors for the contour integral method for 
\(({\bm K}^\eta+\mu{\bm I})^{1/2}\).}
\end{center}
\end{table}

\setlength\fboxsep{1pt}
Finally, Table~\ref{tab:exp} shows the approximation error
\(e_F\big(\exp({\bm K}^\eta)-\boxed{\exp}({\bm K}^\eta)\big)\)
of the matrix exponential for different values of \(N\). 
The true matrix exponential is estimated by a power series of length \(30\) directly applied
to the matrix \({\bm X}\).
Here, we found that the error
starts to stagnate for more than 8 terms in the expansion. The largest eigenvalue satisfies
\(\|{\bm K}^\eta\|_2\approx 0.337\) (estimated by a Rayleigh quotient iteration
with 50 iterations), hence explaining the rapid convergence. 
Note that we do not require
any regularization here, as just matrix products are computed.

\begin{table}[htb]
\begin{center}
\begin{tabular}{|c|r|r|r|r|r|}\hline
$N$ & $5\, 000$ & $10\, 000$ & $50\, 000$ & $100\,000$ & $234\,553$\\\hline
$e_F$ & $1.48\cdot 10^{-9}$& $5.12\cdot 10^{-10}$& $5.51\cdot 10^{-11}$& $4.08\cdot 10^{-11}$& $1.88\cdot 10^{-11}$\\\hline
\end{tabular}
\caption{\label{tab:exp}Errors for the approximation of
\(\exp({\bm K}^\eta)\) by the power series of the exponential.}
\end{center}
\end{table}

\subsection{Gaussian process implicit surfaces}
\begin{figure}[htb]
\begin{center}
\includegraphics[scale =0.085, clip, trim = 860 300 860 300]{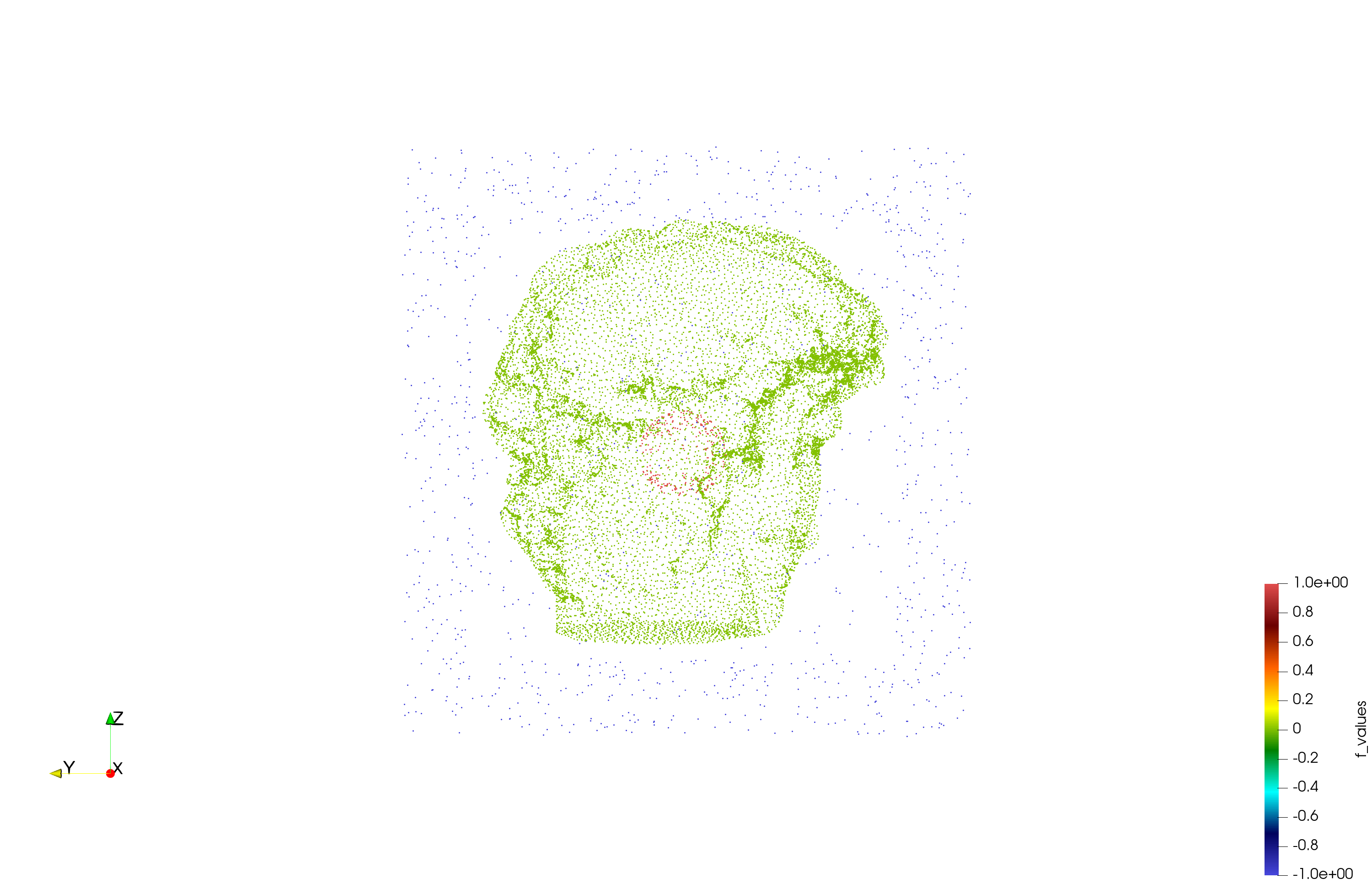}\hfill
\includegraphics[scale =0.085, clip, trim = 860 300 860 300]{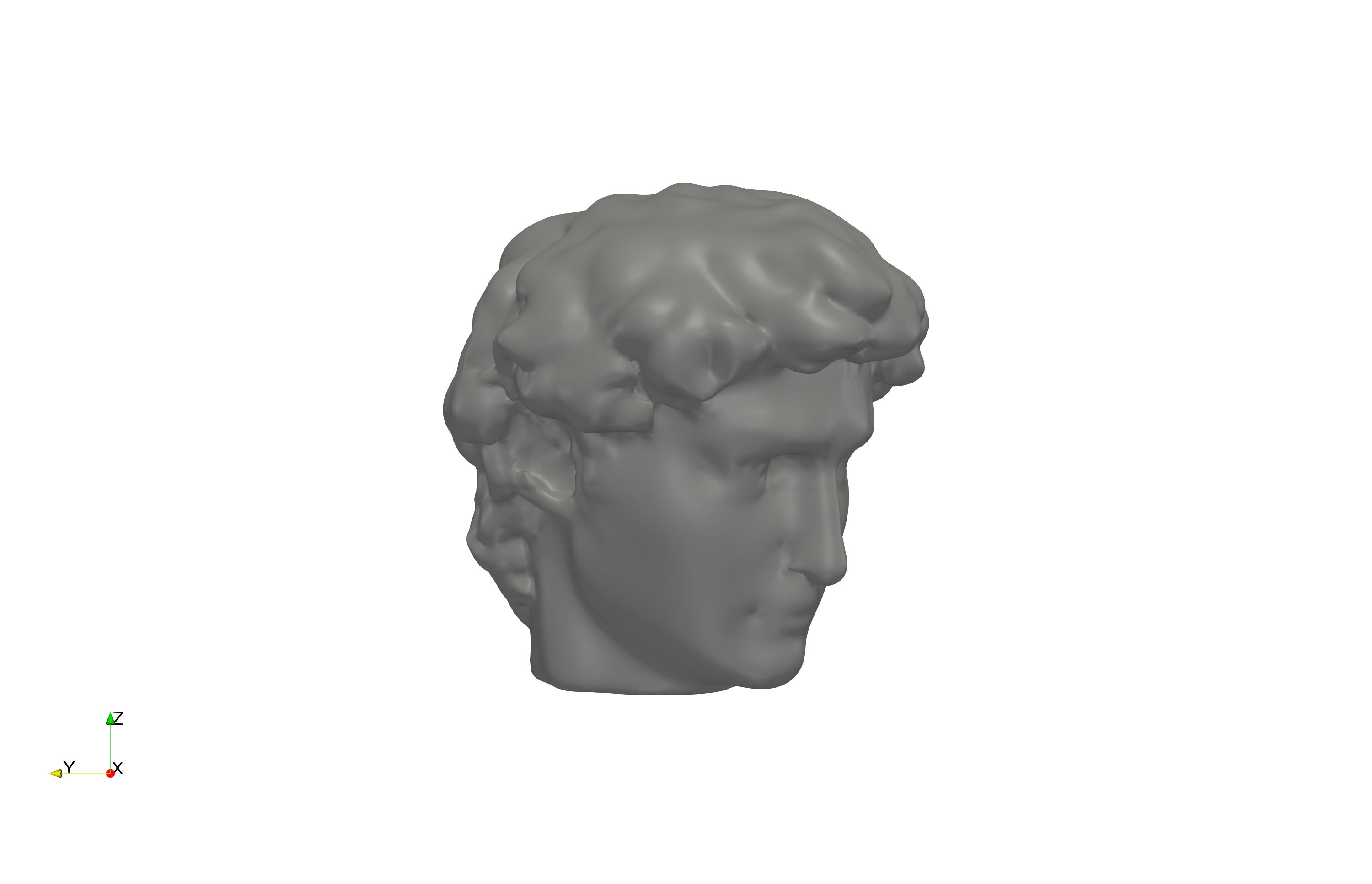}\hfill
\includegraphics[scale =0.079, clip, trim = 800 330 800 180]{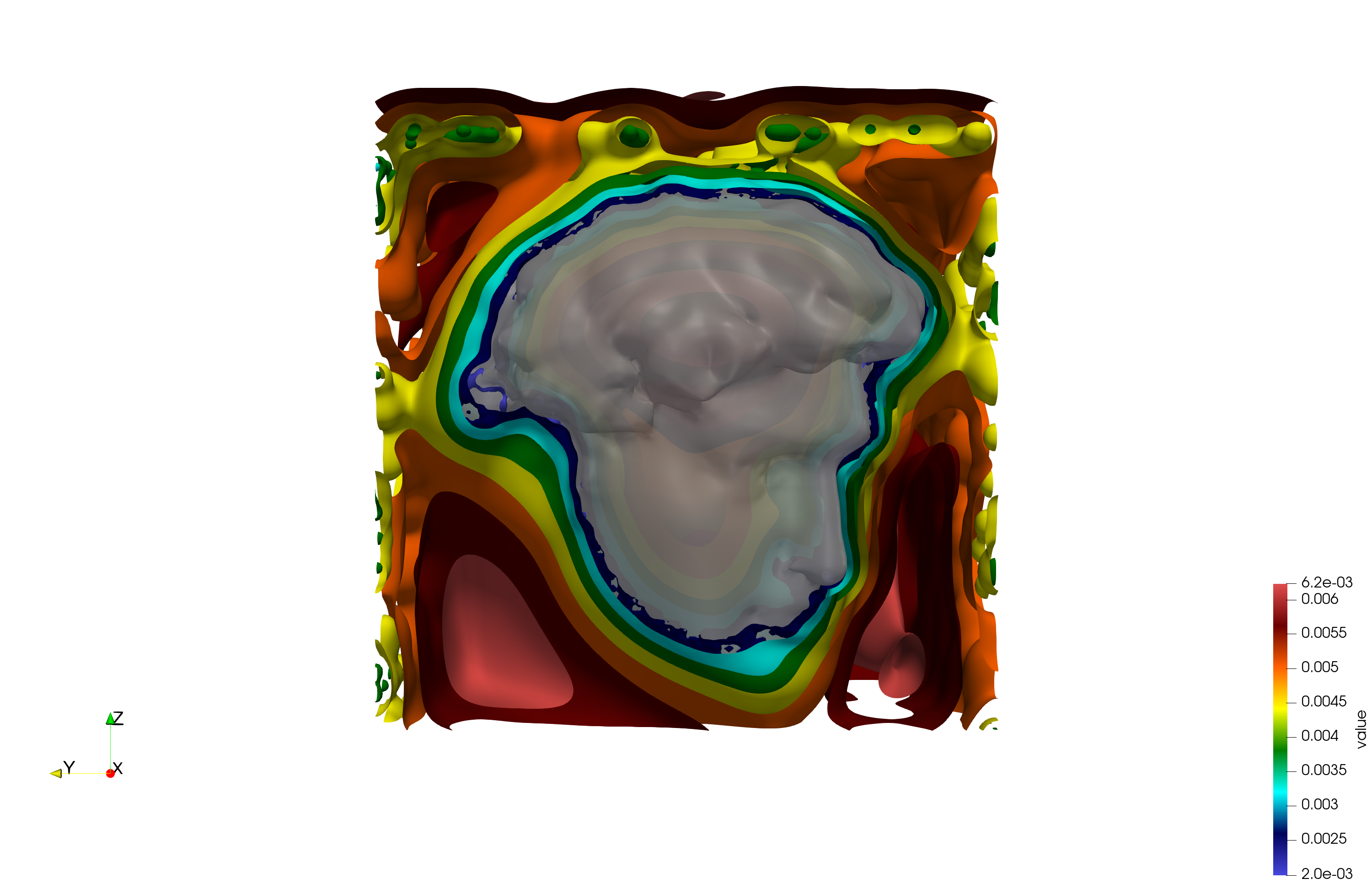}
\caption{\label{fig:davidReconst}Left panel: Data points for the surface reconstruction.
Red corresponds to a value of $1$, green to a value of $0$ and blue to a value of $-1$.
Middle panel: 0-level set of the posterior expectation evaluated at a regular grid.
Right panel: Standard deviation for the reconstruction (blue is small, red is large).}
\end{center}
\end{figure}
We consider Gaussian process learning of implicit surfaces.
In accordance with \cite{WF07}, we consider a closed surface \(S=\partial\Omega\)
of dimension \(d-1\), given by the 0-level set of the function 
\[
f\colon\Rbb^d\to\Rbb,\quad f({\bm x})\begin{cases}=0,& {\bm x}\in S,\\
>0, &{\bm x}\in\Omega,\\
<0,&{\bm x}\in\Rbb^d\setminus\overline{\Omega},
\end{cases}
\]
i.e.,
\[
S=\{{\bm x}\in\Rbb^d:f({\bm x})=0\}.
\]
For the function \(f\), we impose a Gaussian process model with
covariance function given by the exponential kernel
\[
\kernel({\bm x},{\bm y})=\frac{1}{N}e^{-6\|{\bm x}-{\bm y}\|_2}
\]
and prior mean zero.
Then, given the data sites \(X\) of size $N\isdef |X|$ and the 
noisy measurements \({\bm y}=f(X)+{\bm\varepsilon}\), where 
\({\bm\varepsilon}\sim\Ncal({\bm 0},\mu{\bm I})\), the posterior 
distribution for the data sites \(Z\subset\Rbb^3\) is determined by
\begin{align*}
\Ebb[f(Z)|X,{\bm y}]&={\bm K}_{ZX}({\bm K}_{XX}+\mu{\bm I})^{-1}{\bm y},\\
\operatorname{Cov}[f(Z)|X,{\bm y}]&=
{\bm K}_{ZZ}-{\bm K}_{ZX}({\bm K}_{XX}+\mu{\bm I})^{-1}{\bm K}_{ZX}^\intercal.
\end{align*}
Herein, setting \(M\isdef |Z|\), we have \({\bm K}_{XX}=[\kernel(X,X)]
\in\Rbb^{N\times N}\), \({\bm K}_{ZX}=[\kernel(Z,X)]\in\Rbb^{M\times N}\), 
\({\bm K}_{ZZ}=[\kernel(Z,Z)]\in\Rbb^{M\times M}\).

The matrix \({\bm K}_{ZX}\) can efficiently be computed by using one samplet tree for \(Z\)
and a second samplet tree for \(X\), while \(({\bm K}_{XX}+\mu{\bm I})^{-1}\)
can be computed as in the previous examples. 
Hence, the computation of the posterior
mean \(\Ebb[f(Z)|X,{\bm y}]\) is straightforward.
For \(X\), we use samplets with $q+1=4$ vanishing moments, while samplets with $q+1=3$
vanishing moments
are applied for \(Z\). Moreover, we use an a-posteriori threshold of \(\tau=10^{-4}/N\)
for \({\bm K}_{ZX}^{\eta}\).

Similarly, we can evaluate the covariance in samplet coordinates. However, the evaluation
of the standard deviation \(\sqrt{\operatorname{diag}(\operatorname{Cov}[f(Z)|X,{\bm y}])}\)
requires more care. Here we just transform \({\bm K}_{ZX}\) with respect to the points
in \(X\) and evaluate the diagonal resulting in a computational cost of 
\(\mathcal{O}\big(MN\log N\big)\). 

The left panel in Figure~\ref{fig:davidReconst} shows the initial setup. 240 data points with
a value $-1$ are located on a sphere within the point cloud, $15\,507$ points with a value of $0$
are located at its
surface and $1200$ points with a value of $1$ are located on a box enclosing it.
This results in \(N=16\,947\) data points in total. The ridge 
parameter was set to \(\mu=2\cdot10^{-5}\). The conditional expectation and the
standard deviation have been computed on a regular grid with \(M=8\,000\,000\) points.
The middle panel in Figure~\ref{fig:davidReconst} shows the 0-level set while the
right panel shows the standard deviation. As expected, the standard deviation 
is lowest close to the data sites (blue is small, red is large).

\section{Conclusion}
We have presented a sparse matrix algebra
for kernel matrices in samplet coordinates.
This algebra allows for the rapid addition,
multiplication and inversion of (regularized)
kernel matrices, which operations mimic algebras of 
corresponding pseudodifferential operators.
The proposed arithmetic operations extend to 
$S$-formatted, approximate representations of 
holomorphic functions of $S$-formatted 
approximations of self-adjoint operators, which
are likewise realized in log-linear cost.
While the addition is straightforward,
we have derived an error and
cost analysis for the multiplication,
and for the approximate evaluation of holomorphic 
operator-functions in log-linear cost.
The $S$-formatted approximate
inversion is realized by selective inversion
for sparse matrices, which also enables the
computation of general matrix functions by
the contour integral approach. 
The numerical
benchmarks corroborate the theoretical findings
for data sets in two and three dimensions.
As a relevant example from computer graphics,
we have considered Gaussian
process learning for the computation of a signed
distance function from scattered data.

We expect the presently developed fast kernel matrix algebra 
to impact various areas in machine learning and statistics, where
kernel-based approximations appear (e.g.\ \cite{BG19},
\cite{NeurOp20} and the references there).

\appendix
\section{Pseudodifferential operators}\label{sec:pseudos}
We present basic definitions and terminology from the theory
of pseudodifferential operators, in particular elements of the 
calculus of pseudodifferential operators, going back to Seeley
\cite{Seeley66,Seeley69}. We adopt the notation for the statements 
of results on pseudodifferential operators from the monographs of 
H\"ormander \cite{HorIII} and Taylor \cite{Taylor81}, 
but hasten to add that \emph{infinite smoothness of kernels
in the corresponding operator calculi is not essential in 
S-formatted matrix algebra, as the S-compression is based 
on Calder\'on-Zygmund estimates \eqref{eq:kernel_estimate} 
to order $q+1$.}
\subsection{Symbols}\label{sec:Symb}
For an order $r\in\mathbb{R}$ and an open and bounded 
domain $\O\subset\mathbb{R}^d$ with smooth boundary, 
the symbol class $S^r(\O\times\mathbb{R}^d)$ consists of 
functions $a\in C^\infty(\O\times\mathbb{R}^d)$ such that,
for any $K\Subset\O$ and for every ${\bm\alpha},
{\bm\beta}\in\mathbb{N}^d$, there exist constants 
$C_{{\bm\alpha},{\bm\beta}}(K)>0$ such that 
\begin{equation}\label{eq:SymDeriv}
\forall {\bm x}\in K,\ {\bm\xi}\in\mathbb{R}^d:\quad 
\left| \partial_{\bm x}^{\bm\alpha} \partial_{\bm\xi}^{\bm\beta} 
a({\bm x},{\bm\xi}) \right| 
\leq 
C_{{\bm\alpha},{\bm\beta}}(K) \langle{\bm\xi}\rangle^{r-|{\bm\beta}|},
\end{equation}
where $\langle{\bm\xi}\rangle = (1+\|{\bm\xi}\|_2^2)^{1/2}$. 
The class $S^r(\O\times\mathbb{R}^d)$ is contained in the 
H\"ormander class $S^{r}_{1,0}(\O\times\mathbb{R}^d)$; 
we shall not require the general classes $S^r_{\rho,\delta}
(\O\times\mathbb{R}^d)$, cf.\ \cite{HorIII}, and, therefore, 
omit the fine indices.

A function $a_r\in C^\infty(\O\times\mathbb{R}^d\backslash
\{ 0 \})$ is called {\em positively homogeneous of degree $r$} if
\[
\forall {\bm x}\in \Omega,\; \forall t>0, \; {\bm 0}\ne {\bm\xi} \in\mathbb{R}^d: 
\quad a_r({\bm x},t{\bm\xi}) = t^r a_r({\bm x},{\bm\xi}).
\]
Note that then $\chi({\bm\xi}) a_r({\bm x},{\bm\xi})\in 
S^r(\O\times\mathbb{R}^d)$ for any smooth, nonnegative 
cut-off function $\chi$ which vanishes identically for $\|{\bm\xi}\|_2 
\le 1/2$ and $\chi({\bm\xi}) \equiv 1$ for $\|{\bm\xi}\|_2 \geq 1$.
For a symbol $a\in S^r(\O\times\mathbb{R}^d)$, the 
corresponding pseudodifferential operator $\mathcal{A}$ 
is defined for $u\in C_0^\infty(\O)$ via the oscillatory 
integral, cf.\ \cite{HorI},
\begin{equation}\label{eq:DefPsiDO}
\mathcal{A}({\bm x},-i\partial_{\bm x})u({\bm x}) 
= (2\pi)^{-d/2} \int_{{\bm\xi}\in\mathbb{R}^d} 
e^{i\langle {\bm x},{\bm\xi}\rangle} 
a({\bm x},{\bm\xi}) \hat{u}({\bm\xi})\d{\bm\xi},\quad{\bm x}\in \O.
\end{equation}
The set of all pseudodifferential operators $\mathcal{A}$ 
generated via \eqref{eq:DefPsiDO} from a symbol $a\in 
S^r(\O\times\mathbb{R}^d)$ is denoted by $OPS^r(\O)$.

A symbol $a\in S^r(\O\times\mathbb{R}^d)$ is called 
{\em classical symbol of order $r\in\mathbb{R}$} if
for every $k\in\mathbb{N}$ there exist functions 
$a_{r-k}({\bm x},{\bm\xi})\in S^{r-k}(\O\times\mathbb{R}^d)$
such that $a \sim \sum_{k} a_{r-k}$ (in the sense of asymptotic 
expansions of symbols, compare \cite{HorIII}), where $a_{r-k}$ 
is homogeneous of degree $r-k$, i.e., there holds $a_{r-k}({\bm x},t{\bm\xi}) 
= t^{r-k}a_{r-k}({\bm x},{\bm\xi})$ for every $t>0$ and for every 
${\bm\xi}\in\mathbb{R}^d$ with $\|{\bm\xi}\|_2\ge 1$. As a consequence 
of the asymptotic expansion of $a\in S^r_{cl}(\O\times\mathbb{R}^d)$,
for every ${\bm\alpha},{\bm\beta}\in\mathbb{N}^d$ and for every 
$K\Subset\O$ exists a constant $c_{{\bm\alpha},{\bm\beta}}
(K)\in (0,1)$ such that for every $N\in\mathbb{N}$ holds
\begin{equation}\label{eq:AsyClass}
\forall{\bm x}\in K,\ {\bm\xi}\in\mathbb{R}^d:\quad 
\left| 
\partial^{\bm\alpha}_{\bm x}\partial^{\bm\beta}_{\bm\xi} \left(a({\bm x},{\bm\xi}) 
- \sum_{k=0}^N a_{r-k}({\bm x},{\bm\xi})\right) \right|
\leq 
c_{{\bm\alpha},{\bm\beta}}(K)
\langle {\bm\xi}\rangle^{r-N-|{\bm\beta}|-1}.
\end{equation}

\subsection{Calculus}\label{sec:calculus}
Pseudodifferential operators admit calculi which are crucial 
for the subsequent matrix arithmetic. We collect properties 
of the calculi in $S^r_{cl}(\O\times\mathbb{R}^d)$ that are
required throughout the article.

\begin{proposition}\label{prop:PsDoCalc}
\begin{enumerate}
\item
$\Acal\in OPS^r_{cl}$ implies $\Acal^\star \in OPS^r_{cl}$.
\item
$\Acal\in OPS^r_{cl}$ and $\Bcal\in OPS^t_{cl}$ implies
$\Acal+\Bcal\in OPS^{\max\{r,t\}}_{cl}$.
\item
$\Acal\in OPS^r_{cl}$ and  $\Bcal\in OPS^t_{cl}$ implies
$\Acal\circ\Bcal\in OPS^{r+t}_{cl}$.
\item
If $\Acal\in OPS^r_{cl}$ is invertible and elliptic,
then there holds $A^{-1} \in OPS^{-r}_{cl}$.
\end{enumerate}
\end{proposition}
\begin{proof} 
The asserted properties for $OPS^r_{cl}$ are standard properties
for this algebra.
\end{proof}

In case of the Mat\'ern kernels, expanding \eqref{eq:MaternSymbol}
asymptotically, as $\|{\bm {\bm\xi}}\|_2\to \infty$, and comparing with 
\eqref{eq:AsyClass}, we deduce that the associated integral 
operator satisfies $\cK_\nu\in OPS_{cl}^{-2\nu-d}$. It follows also 
from the symbolic calculus in Proposition \ref{prop:PsDoCalc} that 
the inverse $\cK_\nu^{-1}\in OPS_{cl}^{2\nu+d}$. 
Indeed, the symbol of the inverse corresponds to the differential operator 
$\mathcal{A}_\nu = \alpha^{-1}(\id-\frac{\ell^2}{2\nu}\Delta)^{\nu+d/2}$ 
which is of order $2\nu+d$. 
\subsection{Kernels}\label{sec:kernels}
Every continuous function on the cartesian product of
two domains $\Omega_1$ and $\Omega_2$, $\kernel\in 
C(\Omega_1 \times\Omega_2)$, defines an integral operator from
$C(\Omega_2)$ to $C(\Omega_1)$ by the formula 
\begin{equation}\label{eq:defcK}
(\cK\phi)({\bm x}_1) = \int_{\Omega_2}
\kernel({\bm x}_1,{\bm x}_2) \phi({\bm x}_2)\d{\bm x}_2.
\end{equation}
For such kernel functions, 
we have particularly, cf.\ \cite[Eq.\ (5.2.1)]{HorI},
\begin{equation}\label{eq:cKTens} 
\langle\cK\phi,\psi\rangle = \langle\kernel,\psi\otimes \phi\rangle
\quad 
\text{for all}\
\psi \in \dom(\Omega_1),\; \phi \in \dom(\Omega_2),
\end{equation}
where we define the space of test functions
\(\dom(\Omega)\isdef C_0^\infty(\Omega)\) as usual.
The characterization \eqref{eq:cKTens} can be extended 
to distributions $\kernel\in \dom'(\Omega_1\times\Omega_2)$ if $\cK\phi$ 
is allowed to be a distribution. Especially, according to 
the (classical) {\em Schwartz Kernel Theorem}, a (distributional)
kernel corresponds in a one-to-one fashion to
a linear operator and vice versa.

\begin{proposition}[Schwartz Kernel Theorem {\cite[Thm.\ 5.2.1]{HorI}}]
\label{prop:KerThm} 
Every distributional kernel $\kernel\in \dom'(\Omega_1\times \Omega_2)$ 
induces, via \eqref{eq:cKTens}, a continuous, linear map 
from $\dom(\Omega_2)$ to $\dom'(\Omega_1)$. Conversely, for 
every linear map $\cK$, there exists a unique distribution 
$\cK$ such that \eqref{eq:cKTens} holds. The distribution $\kernel$ 
is called {\em (distributional) kernel} of $\cK$.
\end{proposition}

Via the Schwartz Kernel Theorem, every classical pseudodifferential 
operator $\Acal\in OPS^r_{cl}(\O)$ with symbol $a\in S^r_{cl}(\O\times
\mathbb{R} ^d)$ can be written as a (distributional) integral operator 
with (distributional) Schwartz kernel $\kernel_{\Acal}$. 
If the order $r$
of the pseudodifferential operator $\Acal\in OPS^r_{cl}(\O)$
is smaller than $-d$, 
its distributional kernel is continuous and satisfies \eqref{eq:kernel_estimate}.

\subsection*{Acknowledgements}
\noindent HH was funded in parts by the SNSF by the grant 
``Adaptive Boundary Element Methods Using Anisotropic Wavelets'' 
(200021\_192041). MM was funded in parts by the SNSF starting grant 
``Multiresolution methods for unstructured data'' (TMSGI2\_211684).
\bibliographystyle{plain}

\end{document}